\documentclass{article}

\usepackage{epsfig,xspace}
\usepackage{latexsym}
\usepackage[matrix,arrow,curve]{xy}
\xyoption{all}

\usepackage[usenames,dvipsnames,svgnames,table]{xcolor}
\usepackage{amssymb,amsmath,latexsym,amsfonts,amsthm,alltt}
\usepackage[usenames,dvipsnames,svgnames,table]{xcolor}
\usepackage[dvips,usenames]{stmaryrd}
\usepackage{url}
\usepackage{graphicx}
\usepackage{float}
\usepackage{amstext}
\usepackage{cite}
\usepackage{hyperref}
\usepackage{MnSymbol}
\usepackage{relsize}

\newtheorem{defn}{Definition}[section]
\newtheorem{rem}[defn]{Remark}
\newtheorem{thm}[defn]{Theorem}
\newtheorem{lemma}[defn]{Lemma}
\newtheorem{prop}[defn]{Proposition}
\newtheorem{coro}[defn]{Corollary}

\newcommand\crule[3][black]{\textcolor{#1}{\rule{#2}{#3}}}

\newcommand{\ra}{\mbox{$\rightarrow$}}
\newcommand{\lra}{\longrightarrow}
\newcommand{\lla}{\longleftarrow}
\newcommand{\la}{\mbox{$\leftarrow$}}
\newcommand{\onlra}[1]{\stackrel{#1}{\lra}}
\newcommand{\Ra}{\Rightarrow}
\newcommand{\La}{\Leftarrow}

\newcommand{\llra}{\longleftrightarrow}

\newcommand{\midsp}{\;|\;}
\newcommand{\sub}[2]{#1_{{}_{#2}}}
\newcommand{\telos}{\hfill$\leftVdash$}
\newcommand{\type}[1]{{\tt #1}}

\newcommand{\iso}{\backsimeq}

\newcommand{\val}[1]{\mbox{$[\![#1]\!]$}}
\newcommand{\forces}{\Vdash}
\newcommand{\dforces}{\forces^{\!\!\partial}}
\newcommand{\yvval}[1]{\mbox{$(\!|#1|\!) $}}

\newcommand{\infrule}[2]{\frac{\mbox{\rm $#1$}}{\mbox{\rm $#2$}}}
\newcommand{\proves}{\vdash}
\newcommand{\vproves}{\mbox{$\medvert\!\!\!\!\sim$}}

\newcommand{\upv}{\upVdash}
\newcommand{\rperp}{\mbox{${}^{\upv}$}}

\newcommand{\gphi}{{\mathcal  G}(Y)}
\newcommand{\gpsi}{{\mathcal  G}(X)}

\newcommand{\bbox}{\blacksquare}

\newcommand{\lperp}{{}\rperp}

\newcommand{\lbdiamond}{\raisebox{-3.5pt}{\mbox{\Huge {$\filleddiamond$}}}}
\newcommand{\blackdiamond}{\raisebox{-1.5pt}{\mbox{\LARGE {$\filleddiamond$}}}}
\newcommand{\lbbox}{\raisebox{-1.2pt}[0pt][0pt]{\crule[black]{0.27cm}{0.27cm}}\hspace*{1pt}}

\newcommand{\stxb}[2]{\mbox{ST$^\bullet_{#1}(#2)$}}
\newcommand{\styc}[2]{\mbox{ST$^\circ_{#1}(#2)$}}
\newcommand{\stx}[2]{\mbox{ST$_{#1}(#2)$}}
\newcommand{\sty}[2]{\mbox{ST$_{#1}(#2)$}}
\newcommand{\filt}{\mbox{\rm Filt}}

\newcommand{\ufilt}{\mbox{\rm Ult}}
\newcommand{\idl}{\mbox{\rm Idl}}

\newcommand{\lfspoon}{\leftfilledspoon}
\newcommand{\rfspoon}{\rightfilledspoon}

\title{Modal Translation\\ of Substructural Logics}
\author{Chrysafis Hartonas\\
Department of Computer Science and Engineering,\\
University of Applied Sciences of Thessaly (TEI of Thessaly), Greece\\
$\type{hartonas@teilar.gr}$}

\begin{document}
\maketitle

\begin{abstract}
In an article dating back in 1992, Kosta Do\v{s}en initiated a project of modal translations in substructural logics, aiming at generalizing the well-known G\"{o}del-McKinsey-Tarski translation of intuitionistic logic into {\bf S4}. Do\v{s}en's translation worked well for (variants of) {\bf BCI} and stronger systems ({\bf BCW}, {\bf BCK}), but not for systems below {\bf BCI}.  Dropping structural rules results in logic systems without distribution. In this article, we show, via translation, that every substructural (indeed, every non-distributive) logic is a fragment of a corresponding sorted, residuated (multi) modal logic.  At the conceptual and philosophical level, the translation provides a classical interpretation of the meaning of the logical operators of various non-distributive propositional calculi. Technically, it allows for an effortless transfer of results, such as compactness, L\"{o}wenheim-Skolem property and decidability.\\[1mm]
\end{abstract}

\section{Introduction}
\label{intro}
In 1933, Kurt G\"{o}del published a paper \cite{goedel} he had presented the year before in the Mathematics Colloqium in Vienna, chaired by his former teacher, Karl Menger. In his presentation and paper G\"{o}del proposed his well-known interpretation of Intuitionistic  Logic ({\bf IL}) into a certain modal extension of Classical Propositional Logic ({\bf CPL}), now widely known as Lewis's system {\bf S4}. G\"{o}del proposed to translate every {\bf IL} formula $\varphi$ to an {\bf S4}-formula $\varphi^\Box$ by prefixing a $\Box$ to every subformula of $\varphi$. He proved that if {\bf IL}$\;\proves\varphi$, then {\bf S4}$\;\proves\varphi^\Box$ and he advanced the conjecture, later proven by McKinsey and Tarski \cite{MckinseyTarski1948}, that the converse also holds. Grigori Mints \cite{Mints2012} complemented the result by providing a characterization of modal formulas equivalent in {\bf S4} to GMT-translations of intuitionistic formulas. Grzegorczyk \cite{grz} showed that adding to {\bf S4} the axiom (G) $\Box(\Box(p\ra\Box p)\ra p)\ra p$, the (now known as) Grzegorczyk system {\bf Grz} = {\bf S4}+G is strictly stronger than {\bf S4}, it is not contained in {\bf S5} and the same result about the GMT-translation of {\bf IL} into {\bf Grz} is provable. This was also shown to be true for McKinsey's system {\bf S4.1} = {\bf S4}+($\Box\Diamond\varphi\rightarrow\Diamond\Box\varphi$). This quickly led to the investigation of logics, variably called {\em super-intuitionistic}, or {\em intermediate logics} between {\bf IL} and {\bf CPL} and their {\em modal companions}, i.e. modal logics containing {\bf S4}, cf. \cite{kuzne,maksimova1974,chagrov,esakia,bezhan1}, in some cases logics between {\bf S4} and {\bf S5} \cite{dummett}, which culminated in the Blok-Esakia theorem, independently proven by Blok \cite{blok} and Esakia \cite{esakia}. The topological interpretation of {\bf S4} \cite{Mckinsey1944-MCKTAO-6,Mckinsey1946-MCKOCE,MckinseyTarski1948}, interpreting the box operator as an interior operator, provided furthermore the background for investigating dualities between modal companions of superintuitionistic logics and topological spaces \cite{bezhan1}. In all of the above cases distribution is assumed and topological dualities investigated are based on the Priestley \cite{hilary} duality for distributive lattices, or on the Esakia  \cite{esakia,esakia2,esakia3} duality for Heyting algebras.

For logics without distribution, Goldblatt \cite{goldb} provided a translation of Orthologic {\bf O} into the {\bf B} system of modal logic. This quickly led to modal approaches to quantum logic \cite{pessoa,Baltag2011}, as an immediate application. Goldblatt's result is specific to Orthologic and it was Kosta Do\v{s}en \cite{dosen-modal} who embarked in a project of extending the GMT translation to the case of {\em substructural logics}. Do\v{s}en's main result is an embedding of {\bf BCI} (to which he refers as {\bf BC}), {\bf BCK}, {\bf BCW} with intuitionistic negation into {\bf S4}-type modal extensions of versions of {\bf BCI}, {\bf BCK}, {\bf BCW}, respectively, all equipped with a classical-type negation (satisfying double negation and De Morgan laws). The result is established by proof-theoretic means, but the approach does not apply uniformly to substructural logics at large. Indeed, Do\v{s}en points out that difficulties arise when dealing with weak systems such as the non-commutative and/or non-associative full Lambek Calculus {\bf (N)FL} (see \cite{dosen-modal}, Introduction, as well as end of Section 4).

A parallel investigation into {\em subintuitionistic logics} has also developed, initiated by Corsi \cite{corsi-subint} in 1987 and then by Restall \cite{restall-subint} in 1994, followed up by Celani and Jansana \cite{celani-subint-2001} and others. Unlike substructural logics which are proof-theoretically specified, subintuitionistic logics are typically specified in a model-theoretic way, as the logics of Kripke frames for intuitionistic logic with some frame conditions weakened, or dropped. Proof-theoretically, this is reflected as a weakening of the axiomatization of intuitionistic logic. In all cases, however, and as it is typical in mainstream relevance logic \cite{relevance1,relevance2}, distribution is explicitly forced in the axiomatization and we therefore leave these outside the scope of our investigation.
Our focus, instead, is with  modal translations of non-distributive logics, arising as the logics of normal lattice expansions.

Normality for a lattice operator means that it is an additive operator \cite{jt1,jt2}  $f:\mathcal{L}^{i_1}\times\cdots\times\mathcal{L}^{i_n}\lra\mathcal{L}^{i_{n+1}}$ (distributing over binary joins in each argument place), where for each $j$, $i_j\in\{1,\partial\}$ and $\mathcal{L}^\partial$ designates the opposite lattice of $\mathcal{L}$ (usually designated by $\mathcal{L}^{op}$). The tuple $\delta=(i_1,\ldots,i_n;i_{n+1})$ will be referred to as its {\em distribution type}. For example, boxes and diamonds are normal operators $\Diamond:\mathcal{L}\lra\mathcal{L}$, $\Box:\mathcal{L}^\partial\lra\mathcal{L}^\partial$. Implication is a normal operator with distribution type $(1,\partial;\partial)$, in other words it is an additive operator $\ra:\mathcal{L}\times\mathcal{L}^\partial\lra\mathcal{L}^\partial$.

In this article, we restrict attention to the hierarchy of substructural logic systems displayed in Figure \ref{sub-hierarchy}
\begin{figure}[!htbp]
\caption{Hierarchy of Substructural  Systems}
\label{sub-hierarchy}
{\small
\[
\xymatrix{
 &&& {\bf BCK}\\
{\bf NFL}\ar@{-}[r] & {\bf FL}\ar@{-}[r] & {\bf BCI}\ar@{-}[r]\ar@{-}[ur] & {\bf BCW}
}
\]
}
\end{figure}
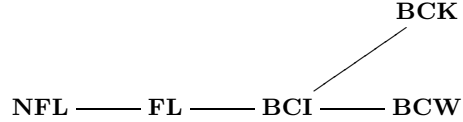
and we present a uniform translation of the logics, proven here to be full and faithful. The target system of the translation is a corresponding sorted, residuated (multi) modal logic. Though we do not treat here Grishin's dual operators \cite{grishin}, it will be obvious that our approach is easily extendible to the Full Lambek-Grishin ({\bf FLG}) system.

The main structure of the argument is the following.  We insist that a substructural logic {\bf SUB} and its modal companion logic {\bf sub.ML}$_2$ are interpreted over the same class $\mathbb{F}$ of frames $\mathfrak{F}$. For modal definability (of the frames for substructural logics) reasons, we remove restrictions typically assumed to hold, such as  frames being separated and reduced (RS-frames \cite{mai-gen}). Models $(\mathfrak{F},\forces)$, $(\mathfrak{F},\models)$ are defined for {\bf SUB} and {\bf sub.ML}$_2$, respectively, and a translation $\varphi^\bullet$ and co-translation $\varphi^\circ$ of sentences of the language of {\bf SUB} into the sorted language of {\bf sub.ML}$_2$ are defined. We then prove that $\varphi\forces\psi$ in {\bf SUB} iff $\varphi^\bullet\models\psi^\bullet$ iff $\psi^\circ\models\varphi^\circ$ in {\bf sub.ML}$_2$. To transfer the result on the provability relation, completeness theorems for each of {\bf SUB} and {\bf sub.ML}$_2$ need to be proven, hence obtaining that $\varphi\proves\psi$ in the proof system for {\bf SUB} iff $\varphi^\bullet\proves\psi^\bullet$ iff $\psi^\circ\vproves\hskip1.5pt\varphi^\circ$ in the sorted proof system for {\bf sub.ML}$_2$, thereby establishing that the translation is full and faithful.

Lacking distribution, substructural logics are interpreted \cite{suzuki8,mai-gen,mai-grishin} in sorted frames $\mathfrak{F}=(X,\upv,Y)$ with additional structure (to account for the residuated operators). As we detail in Section \ref{sorted section}, the dual complex algebra of the frame can be taken to be either the formal concept lattice of the frame (equivalently, its complete lattice $\mathcal{G}(X)$ of Galois-stable sets), or the sorted, residuated modal algebra $\largediamond:\powerset(X)\leftrightarrows\powerset(Y):\lbbox$ associated to the frame, where the residuated operators are generated by the complement $I$ of the Galois relation $\upv$ of the frame. Additional frame relations generate operators both on powerset algebras and on sets of Galois stable sets where, roughly, the latter are the Galois closures of the former (see \cite{sdl-exp,discr,discres}) and exploiting this connection leads to the possibility for a modal translation of substructural (non-distributive, more generally) logics.

Logically, the sorted, residuated modal logics we consider in this article arise as a synthesis in the sense of \cite{synthesis} of classical propositional logic and the logic of residuated Boolean algebras, combined together by means of a logical adjunction (a pair of sorted, residuated operators).

From an algebraic point of view, we consider the logics of sorted, residuated modal algebras $\Diamond:\mathcal{A}\leftrightarrows\mathcal{B}:\bbox$, $a\leq\bbox b$ iff $\Diamond a\leq b$, where both $\mathcal{A,B}$ are Boolean algebras and where $\mathcal{A}$ is moreover a residuated Boolean algebra \cite{residBA} with operators $\lfspoon,\odot,\rfspoon$, i.e. $a_1\leq c\lfspoon a_2$ iff $a_1\odot a_2\leq c$ iff $a_2\leq a_1\rfspoon c$. If the full Lambek-Grishin calculus is to be considered, then both $\mathcal{A,B}$ are residuated Boolean algebras. These structures serve as algebraic models to the sorted, residuated modal logics we consider in this article. Note that the set $\mathcal{A}_\bullet=\{\bbox b\midsp b\in\mathcal{B}\}$ is a meet sub-semilattice of $\mathcal{A}$ which has a full lattice structure (non-distributive, in general) when defining joins in $\mathcal{A}_\bullet$ by setting $a\veebar a'=\bbox(\Diamond a\vee\Diamond a')$. Similarly for the set $\mathcal{B}_\circ=\{\Box a\midsp a\in\mathcal{A}\}$ and they are dually isomorphic $\mathcal{A}_\bullet\iso\mathcal{B}^\partial_\circ$. In the case of interest in this article, the closure operator $\bbox\Diamond$ carries over the residuation structure of $\mathcal{A}$ to $\mathcal{A}_\bullet$.  More precisely, our representation results \cite{sdl-exp,discres} have established the following:

\begin{thm}[\!\!\cite{sdl-exp,discres}]\rm
Every residuated lattice $\mathcal{L}=(L,\wedge,\vee,0,1,\la,\circ,\ra)$ is a sublattice of the lattice $\mathcal{A}_\bullet$ of a sorted, residuated modal algebra \mbox{$\Diamond:\mathcal{A}\leftrightarrows\mathcal{B}:\bbox$} where, moreover, $(\mathcal{A},\lfspoon,\odot,\rfspoon)$ is a residuated Boolean algebra and where the lattice operator $\circ$ is the closure of the restriction of $\odot$ on $\mathcal{A}_\bullet$, while its residuals $\la,\ra$ are simply the restrictions of the residuals \mbox{$\lfspoon,\rfspoon$ of $\odot$ on $\mathcal{A}_\bullet$.}\telos
\end{thm}

The target system of the translation can be therefore taken to be a corresponding sorted, residuated (multi) modal logic. Options exist as to the specification of the signature of the target companion modal logic and, in accordance with the previous discussion, we opt  for  the following sorted modal language
\begin{eqnarray*}
\alpha &:=& P_i\;(i\in\mathbb{N})\midsp\neg \alpha\midsp \alpha\wedge \alpha\midsp \lbbox \beta\midsp \alpha\odot \alpha\midsp \alpha\lfspoon \alpha\midsp \alpha\rfspoon \alpha \\
\beta &:=& Q_i\;(i\in\mathbb{N})\midsp\neg \beta\midsp \beta\wedge \beta\midsp \largesquare \alpha\midsp \beta\ostar\beta \midsp \beta\leftspoon\beta\midsp\beta\rightspoon\beta
\end{eqnarray*}
for the full Lambek-Grishin ({\bf FLG}) calculus, where $\odot,\ostar$ are residuated binary diamonds, with residuals $\lfspoon,\rfspoon$ and $\leftspoon,\rightspoon$, respectively. However, for the purposes of this article, focusing on the systems of Figure \ref{sub-hierarchy}, we restrict to the sublanguage without the $\leftspoon,\ostar,\rightspoon$ logical operators.

As an application for the usefulness of the result, we extend the standard translation of modal logic into first-order logic to the two-sorted case. Using established results \cite{enderton,FOL-mortimer,FOL-2var,FOL-decision} for sorted first-order logic and for its  two-variable fragment we obtain compactness, finite model property and decidability, as well as the L\"{o}wenheim-Skolem property for just about any substructural logic. A complexity bound for the satisfiability problem of the two-variable fragment of first-order logic is also known \cite{FOL-decision}, which can be probably improved for the particular fragments at hand.

The structure of this article is as follows.

Section \ref{prelims} briefly goes through some necessary preliminaries, presenting the basics of sorted frames, running through a quick presentation of substructural logics and reminding the reader of the basics on canonical extensions of lattice expansions.

Section \ref{rel-section} studies sorted frames with relations and their associated complex algebra (in fact, algebras (plural), as both a sorted algebra and an algebra of Galois stable sets are of significance) and it draws on recent work by this author on representation and duality for normal lattice expansions \cite{sdl-exp,pnsds,discr,discres}.

Section \ref{modal-t-section} specifies the proposed modal translation and the argument that the translation is full and faithful is detailed. The main structure of the argument is that given a substructural logic {\bf SUB} and its sorted modal companion logic {\bf sub.ML}$_2$ a translation $\varphi^\bullet,\psi^\bullet$ of {\bf SUB}-sentences $\varphi,\psi$ is defined and it is shown that $\varphi^\bullet\models\psi^\bullet$ in some {\bf sub.ML}$_2$-model $\mathfrak{M}$ iff $\varphi\forces\psi$ in an associated {\bf SUB}-model $\mathfrak{N}$ (produced by just tampering with the interpretation function of $\mathfrak{M}$ and allowing only Galois stable sets to be in its range). To conclude from this that the sequent $\varphi\proves\psi$ is provable in {\bf SUB} iff its translation $\varphi^\bullet\proves\psi^\bullet$ is provable in the companion logic {\bf sub.ML}$_2$, we establish soundness and completeness theorems  of both {\bf SUB} and {\bf sub.ML}$_2$ over the {\em same} class of frames.

Some consequences regarding properties of the logics we consider are then presented in Sections \ref{fol trans} and \ref{sort reduction section}. Pointers for generalization of the results to just about any non-distributive logic are suggested in the Conclusions section \ref{conc-section}.

\section{Review and Technical Preliminaries}
\label{prelims}
\subsection{Sorted Frames}
\label{sorted section}
Single-sorted modal systems cannot express a fact such as that the set $W$ of worlds (states, points) is partitioned in a number of subsets, which is typical of domains populated by distinct {\em types} of individuals. Sorted modal logic arises precisely when taking into consideration the typing of the domain of the frame.
Many-sorted modal logic has so far received little attention. Yde Venema's \cite{yde-points-lines}, preceded by Balbiani's  modal multi-logic of geometry \cite{balbiani-multilogic}, employs a two-sorted system whose intended application domain is the projective plane, the two sorts being points and lines. Steven Kuhn's  PhD dissertation \cite{kuhn-sorted}, supervised by Kit Fine, is perhaps the only source where a comprehensive study of many-sorted modal logics is carried out, dating back in 1976 and with a lot having changed in the meantime, as far as the theory of modal logic is concerned. Venema and Balbiani's work is developed in the context of spatial logics, where individuals of distinct sorts populate the application domain of interest. Two-sorted frames have been the focus of attention in other contexts, as well, such as Wille's formal concept analysis (FCA) \cite{wille} and Hardegree's logic of natural kinds \cite{hardegree}, but they were not accompanied by an interest in the two-sorted modal logic of the frames, though the approach seems to suit the needs of FCA extensions such as temporal FCA \cite{wolff}, or rough FCA \cite{kent}. Explicitly related to modal logics are also the sorted frames of  \cite{bezhan2,bezhan3}, called {\em step-frames}, which are structures $(X,Y,f,R)$ where $f:X\lra Y$ and $R\subseteq X\times\powerset(Y)$, with a sorted diamond operator $\sub{\lozenge}{R}:\powerset(Y)\lra\powerset(X)$, used in defining approximations of free algebras of modal logics.

The class of {\em two-sorted frames} we consider consists of triples $(X,\upv, Y)$, where $X,Y$ are non-empty sets and $\upv\;\subseteq X\times Y$ is a binary relation (to be called the {\em Galois relation} of the frame). Two-sorted frames may be concretely regarded as {\em information systems}, or {\em object-attribute systems}, where $X$ is the set of formal objects, $Y$ is the set of formal attributes of objects and $\upv$ is the attribute assignment map.
In general, a frame will have additional relations $R_i, i\in I$, of various sorting (contained in some product with factors  $X$, $Y$).

Sorted frames generalize plain Kripke frames. Indeed, if $X=Y$ and $\upv$ is the non-identity relation $\neq$, then both maps of the Galois connection \eqref{galois} are the set-complementation operator $\lperp U=-U=U\rperp$.

The Galois relation of the frame generates a Galois connection (hence its naming)
\begin{eqnarray}
U\rperp&=\{y\in Y\midsp\forall x\in U\; x\upv y\} & \lperp V\;=\{x\in X\midsp \forall y\in V\;x\upv y\}\label{galois}
\end{eqnarray}
and, by composition of the Galois maps, a closure operator on each of $\powerset(X)$ and $\powerset(Y)$.  The relation $x\preceq z$ iff $\{x\}\rperp\subseteq\{z\}\rperp$ iff $\forall y\in Y\;(x\upv y\lra z\upv y)$ is easily seen to be a preorder on $X$. Similarly for $Y$. A frame is {\em separated} ({\em clarified}) \cite{wille2} if the preorder $\preceq$ is a partial order. Separation turns out to be necessary in order to prove discrete, or Stone  dualities for various classes of lattices with operators \cite{discr,discres}, but it will not be assumed for frames in this article. The reason for this is  that our argument is mostly semantically based in that it requires that we interpret both the substructural and the sorted modal logic systems over the same class of frames and we wish to ensure modal definability of the frames (by the modal companion logics of the substructural systems we examine).

The closure operators can be equivalently regarded as obtained by composition of the residuated pair of maps generated by the complement $I$ of the attribute assignment map $\upv$, i.e. $\lperp(U\rperp)=\lbbox\largediamond U$ and $(\lperp V)\rperp=\largesquare\lbdiamond V$, where $\lperp V=\lbbox(-V), U\rperp=\largesquare(-U)$ and $\largesquare=-\largediamond-,\; \lbdiamond=-\lbbox-$,
\begin{eqnarray}
\largediamond U&=&\{y\in Y\midsp\exists x\in X\;(xI y\wedge x\in U)\}\nonumber \\
\lbbox V&=&\{x\in X\midsp \forall y\in Y\;(xI y\lra y\in V)\}\label{resid}
\end{eqnarray}
Where $I$ is the complement of $\upv$, the following conditions are assumed.
\begin{equation}
\forall x\in X\exists y\in Y\;xIy\hskip1.5cm \forall y\in Y\exists x\in X\;xIy
\label{D-cond}
\end{equation}
The frame conditions \eqref{D-cond} correspond to the inclusions $\lbbox B\subseteq\lbdiamond B$ and $\largesquare A\subseteq\largediamond A$. This implies that $\largesquare\emptyset=\largesquare(-X)=X\rperp=\emptyset$ and, similarly, $\lbbox\emptyset=\emptyset$.

A subset $A$ of $X$ is {\em Galois stable}, or just {\em stable}, if $\lperp(A\rperp)=A$ (equivalently, if $A=\lbbox\largediamond A$). Dually, a subset $B$ of $Y$ is {\em co-stable} if $(\lperp B)\rperp= B$ (equivalently, $B=\largesquare\lbdiamond B$). Stable and co-stable sets are $\preceq$-increasing. Indeed, if $x\in A$ and $x\preceq z$, i.e. $\{x\}\rperp\subseteq \{z\}\rperp$, then from $\{x\}\subseteq A$ we obtain $A\rperp\subseteq \{x\}\rperp\subseteq \{z\}\rperp$, hence $z\in\lperp(\{z\}\rperp)\subseteq A$, by stability of $A$. Similarly for co-stable sets.

A {\em formal concept} \cite{wille2} is a pair consisting of a stable and a co-stable set $(A,B)$ such that $B=A\rperp$ (and then also $A=\lperp B$). The {\em formal concept lattice} of the frame is the complete lattice of formal concepts, with operations defined by
\[
\bigwedge_{i\in I}(A_i,B_i)=(\bigwedge_{i\in I} A_i,\bigvee_{i\in I}B_i)\hskip2cm \bigvee_{i\in I}(A_i,B_i)=(\bigvee_{i\in I}A_i,\bigwedge_{i\in I}B_i)
\]
where meets and joins on the right hand side of the equations are taken in the complete lattice of (co)stable sets, where meets are set intersections and joins are closures of unions. Clearly, the formal concept lattice is isomorphic to the lattice of concept extents (the complete lattice $\gpsi$ of Galois stable subsets of $X$) and dually isomorphic to the lattice of concept intents (the complete lattice $\gphi$ of the co-stable subsets of $Y$).

For a frame $\mathfrak{F}=(X,\upv,Y)$, let $\mathfrak{Con}(\mathfrak{F})$ be its formal concept lattice and $\mathfrak{Res}(\mathfrak{F})$ be its associated sorted, residuated modal algebra.
It is evidently only a matter of focus which algebraic structure we decide to consider as the {\em complex algebra} $\mathfrak{F}^+$ of the frame $\mathfrak{F}$. The two algebraic structures are not unrelated and, in fact, $\mathfrak{Con}(\mathfrak{F})$ is isomorphic to $\mathcal{A}_\bullet$ and it therefore sits inside $\mathfrak{Res}(\mathfrak{F})$ as a meet sub-semilattice of $\mathfrak{Res}(\mathfrak{F})$.

\begin{rem}[Notational Conventions]\rm
\label{notation1}
For a sorted frame $(X,\upv,Y,\ldots)$, perhaps with additional structure, $\upv\;\subseteq X\times Y$ always designates the Galois relation. $(\;)\rperp, \lperp(\;)$ designate the maps of the induced Galois connection \eqref{galois}. The complement of $\upv$ will be consistently designated by $I$ and it generates a pair of residuated set-operators $\largediamond:\powerset(X)\leftrightarrows\powerset(Y):\lbbox$.

$\preceq$ is the preorder on each of $X,Y$ defined by $\{x\}\rperp\subseteq\{z\}\rperp$, for $x,z\in X$ and, similarly, $v\preceq y$ iff $\lperp\{v\}\subseteq\lperp\{y\}$, for $y,v\in Y$.

We typically and consistently use $x,z$ to range over $X$ and $y,v$ to range over $Y$.

For  $u$ in $X$, or in $Y$, we let $\Gamma u$ be the $\preceq$-upset, $\Gamma u=\{w\midsp u\preceq w\}$. Without separation in the frame, $\Gamma u=\Gamma u'$ only implies that $u\preceq u'\preceq u$, without forcing that $u=u'$.

Given a (two-sorted) frame $(X,\upv,Y)$, a sorted frame relation $R$ is any relation $R\subseteq \prod_{i=1}^{n+1}Z_i$, where for each $i$, $Z_i\in\{X,Y\}$. To keep track of the various relations and their sorting we make the convention to indicate the sorting by an upper index, a string of $1$'s and $\partial$'s, with 1 pointing to $X$ and $\partial$ pointing to $Y$. Thus, for example, $R^{11\partial}$ is a relation contained in the cartesian product $X\times(X\times Y)$.
An $(n+1)$-ary relation $R\subseteq Z\times\prod_{i=1}^n Z_i$ will  be always written in the form $zRz_1\cdots z_n$.

The sets obtained by not filling one argument place will be written in the form $zRz_1\cdots z_{i-1}(\;)z_{i+1}\cdots z_n=\{z_i\midsp zRz_1\cdots z_i\cdots z_n\}$. For a two place relation we write $xR$ and $Rz$, rather than $xR(\;)$ and $(\;)Rz$, respectively. Similarly, for a 3-place relation we write $Rzu,xRz$, rather than $(\;)Rzu$ and $xRz(\;)$, respectively, but we use $xR(\;)u$ when the second argument place is not filled.
\end{rem}

\subsection{A Substructural Review}
\label{sub review}
Substructural logics arise from the Gentzen system {\bf LJ} for intuitionistic logic by dropping a combination of the structural rules of exchange, contraction and weakening  and expanding the logical signature of the language to include the operator symbols $\circ$ (fusion, cotenability), $\la$ (reverse implication) and, often, a constant $\mathfrak{t}$.

We let {\bf NFL} be the system with all structural rules dropped (including  the association rule), which is precisely the non-associative Full Lambek calculus.
We write {\bf FL} for the associative Lambek calculus.
For $r\subseteq\{c,e,w\}$ we designate by ${\bf (N)FL}_r$ the system resulting by adding to {\bf (N)FL} the structural rules in $r$ (where $c$ abbreviates `contraction', $e$ abbreviates `exchange' and similarly for $w$ and `weakening'). With the exception of ${\bf FL}_{ecw}$, which is precisely {\bf LJ}, distribution of conjunctions over disjunctions and conversely does not hold,
unless explicitly added to the axiomatization (an option typically taken by mainstream relevance logicians \cite{relevance1,relevance2}).

The algebraic semantics of these systems has been investigated by Hiroakira Ono, Nick Galatos  \cite{ono-galatos,ono1,ono2,ono3,ono4} and several others.

An {\bf FL}-algebra is a structure $\langle L,\leq,\wedge,\vee,0,1,\la,\circ,\ra,\mathfrak{t}\rangle$ where
\begin{enumerate}
  \item $\langle L,\leq,\wedge,\vee,0,1\rangle$ is a bounded lattice
  \item $\langle L,\leq,\circ,\mathfrak{t}\rangle$ is a partially-ordered monoid ($\circ$ is monotone and associative
      and $\mathfrak{t}$ is a two-sided identity element $a\circ\mathfrak{t}=a=\mathfrak{t}\circ a$)
  \item $\la,\circ,\ra$ are residuated, i.e. $a\circ b\leq c$ iff $b\leq a\ra c$ iff $a\leq c\la b$
  \item for any $a\in L$, $\;a\circ 0=0=0\circ a$
\end{enumerate}

An {\bf FL}-algebra is commonly referred to as a {\em residuated lattice} \cite{galatos-book}. {\bf FL}-algebras (residuated lattices) are precisely the algebraic models of the associative full Lambek calculus. By residuation, the distribution types of the operators are $\delta(\circ)=(1,1;1), \delta(\ra)=(1,\partial;\partial)$ and $\delta(\la)=(\partial,1;\partial)$.

An {\bf FL}$_e$-algebra adds to the axiomatization the exchange (commutativity) axiom $a\circ b = b\circ a$ for the cotenability operator (in which case $\la$ and $\ra$ coincide). It corresponds to the system {\bf BCI} studied by Do\v{s}en \cite{dosen-modal}.

An {\bf FL}$_{ew}$-algebra adds the weakening axiom $b\circ a\leq a$, as well, in which case combining with commutativity $a\circ b\leq a\wedge b$ follows. In addition, by $1\circ\mathfrak{t}\leq\mathfrak{t}$, the identity $\mathfrak{t}=1$ holds in {\bf FL}$_{ew}$-algebras.   {\bf FL}$_{ew}$-algebras are also referred to in the literature as {\em full {\bf BCK}-algebras},  corresponding to full {\bf BCK}-logic, resulting from {\bf BCK} whose purely implicational signature is expanded to include conjunction and disjunction connectives, alongside the cotenability logical operator and the constants $0,1$. Algebraically, they constitute the class of {\em commutative integral residuated lattices} \cite{galatos-book}.

Similarly, an {\bf FL}$_{ec}$-algebra, adds the contraction axiom $a\wedge b\leq a\circ b$ (a distinguished case of which is square increasingness $a\leq a\circ a$). {\bf FL}$_{ec}$-algebras are also known as {\bf BCW}-algebras.

The language of the Lambek Calculus (associative, or not) that we consider  is displayed below,
\[
L\ni\varphi\;:=\; p_i\;(i\in\mathbb{N})\midsp\top\midsp\bot\midsp
\varphi\wedge\varphi\midsp\varphi\vee\varphi\midsp\varphi\la\varphi\midsp\varphi\circ\varphi \midsp\varphi\ra\varphi
\]
where we omit the constant $\mathfrak{t}$, avoiding matters of secondary interest.
The use of the names {\bf BCI}, {\bf BCW} and {\bf BCK} is more widespread than the use of the names {\bf FL}$_e$, {\bf FL}$_{ec}$, {\bf FL}$_{ew}$ due to their connection with type theory and the Curry-Howard isomorphism, where {\bf BCI}, {\bf BCK}, {\bf BCW} designate restricted families of $\lambda$-terms and the combinators B,C,I,K,W correspond to the following Hilbert-style axiomatization of their pure implicational fragment.
\begin{tabbing}
({\bf B}) \hskip1cm\= $(\varphi\ra \psi)\ra((\vartheta\ra \varphi)\ra(\vartheta\ra \psi))$   \\
({\bf C}) \>   $(\varphi\ra(\psi\ra \vartheta))\ra(\psi\ra(\varphi\ra \vartheta))$ \\
({\bf I}) \> $\varphi\ra\varphi$\\
({\bf K}) \>    $\varphi\ra(\psi\ra \varphi)$ \\
({\bf W}) \>  $(\varphi\ra(\varphi\ra \psi))\ra(\varphi\ra \psi)$
\end{tabbing}

Proof systems for substructural logics can be found in the literature and we do not take the space to restate them here. The reader may wish to consult \cite{suzuki8}, for a  Gentzen System, or  \cite{dosen-modal} for a Hilbert-style axiomatization. In fact, for the purposes of this article we will be content with a proof system as a symmetric consequence system, with sequents of the form $\varphi\proves\psi$, directly encoding the corresponding algebraic specification.

Fragments of these systems are also of independent interest, corresponding for example to the logics of implicative $(L,\leq,0,1,\wedge,\vee,\ra)$, or bi-implicative lattices $(L,\leq,0,1,\wedge,\vee,\la,\ra)$.
Going in the other direction, extensions of these systems for example by various negation and/or modal operators are also of interest, but we shall not treat such extensions in the present article.

Contrasting to relational semantics, algebraic semantics and soundness and completeness for the systems are well established results in the substructural logics community, see for example \cite{ono4,galatos-book}.

Lacking distribution, substructural logics are interpreted in sorted frames $(X,\upv, Y, R^{111})$, where $R^{111}$ is a ternary relation with the indicated sorting type and $\upv\;\subseteq X\times Y$ is the {\em Galois relation} of the frame. Fusion and implication have been modeled in sorted frames by Gehrke \cite{mai-gen} and then also by Dunn, Gehrke and Palmigiano \cite{dunn-gehrke}, using RS-frames (separated (clarified) and reduced frames, in Wille's sense \cite{wille2}). For our argument of a full and faithful modal translation to go through we need to ensure that the class of frames we use is modally definable, since both the substructural systems and their companion modal logics are to be interpreted in the same class of frames, with respect to which completeness theorems should be derivable. For this reason we drop both restrictions of being separated and being reduced for frames.

Given a frame $\mathfrak{F}=(X,\upv,Y,R^{111})$ a model $\mathfrak{N}=(\mathfrak{F},V)$ for the interpretation of the language of a substructural logic is equipped with a map $V:\{p_i\midsp i\in\mathbb{N}\}\lra\gpsi$, assigning a stable set (a concept extent) to propositional variables. An interpretation $\overline{V}$ of sentences $\varphi$ as formal concepts $\overline{V}(\varphi)=(\val{\varphi},\yvval{\varphi})$ with extent $\val{\varphi}$ and intent $\yvval{\varphi}=\val{\varphi}\rperp$ is then recursively defined. The satisfaction $\forces$ and co-satisfaction $\dforces$ relations are defined by setting, for $x\in X$ and $y\in Y$, $x\forces\varphi$ iff $x\in\val{\varphi}$ and $y\dforces\varphi$ iff $y\in\yvval{\varphi}$. The recursive definition of the interpretation $\overline{V}$ is specified by the appropriate clauses for $\forces$ and $\dforces$, shown in Table \ref{sat}.  The relation $R^{\partial 11}$ involved in the semantic clauses is defined by setting $R^{\partial 11}zz'=(R^{111}zz')\rperp$.

\begin{table}[!htbp]
\caption{(Co)Satisfaction relations for Substructural Logics}
\label{sat}
\begin{tabbing}
$x\forces p_i$\hskip8mm\=iff\hskip3mm\= $x\in V(p_i)$\\
$x\forces\top$ \>iff\> $x=x$\\
$y\dforces\bot$\>iff\> $y=y$\\
$x\forces\varphi\wedge\psi$\>iff\> $x\forces\varphi$ and $x\forces\psi$\\
$y\dforces\varphi\vee\psi$\>iff\> $y\dforces\varphi$ and $y\dforces\psi$\\
$x\forces\varphi\ra\psi$\>iff\> $\forall z\in X\;\forall y\in Y\;(z\forces\varphi\;\wedge\;y\dforces\psi\;\lra\;yR^{\partial 11}zx)$\\
$y\dforces\varphi\circ\psi$\>iff\> $\forall x,z\in X\;(z\forces\varphi\;\wedge\;x\forces\psi\;\lra\;yR^{\partial 11}zx)$\\
$x\forces\psi\la\varphi$ \>iff\> $\forall y\in Y\;\forall z\in X\;(y\dforces\psi\;\wedge\;z\forces\varphi\;\lra\; yR^{\partial 11}xz )$
\end{tabbing}
\end{table}
For a sentence $\varphi$ it suffices to specify either a satisfaction, or a co-satisfaction clause, since in the interpretation $\overline{V}(\varphi)=(\val{\varphi},\yvval{\varphi})$ we have $\val{\varphi}=\lperp\yvval{\varphi}$ and  $\yvval{\varphi}=\val{\varphi}\rperp$. Hence both $x\forces\varphi$ iff \mbox{$\forall y\in Y\;(y\dforces\varphi\;\lra\;x\upv y)$,} as well as  $y\dforces\varphi$ iff $\forall x\in X\;(x\forces\varphi\;\lra\;x\upv y)$ always obtain.

Table \ref{constraints} displays relational constraints on frames, to be considered in the sequel.

\begin{table}[!htbp]
\caption{Relational Constraints on Frames}
\label{constraints}
\begin{enumerate}
\item[(C1)]
For all points $w, z, u, v\in X$, statements (1) and (2) below are equivalent
    \begin{itemize}
      \item[(1)] $\exists x\;(xR^{111}uv\;\wedge\;zR^{111}xw)$
      \item[(2)] $\exists x\;(xR^{111}vw\;\wedge\;zR^{111}ux)$
    \end{itemize}
\item[(C2)] $\forall x,z,z'\in X\;(xR^{111}zz'\;\llra\; xR^{111}z'z)$
\item[(C3)] $\forall x,z,z'\in X\;(xR^{111}zz'\;\lra\;z'\preceq x)$
\item[(C4)] $\forall x\in X\;xR^{111}xx$
\end{enumerate}
\end{table}

The reader familiar with the literature on relevance logic will recognize them as the standard frame conditions relevance logicians consider on frames \cite{relevance1,relevance2,relevance3} and some of our related proof obligations come across the same type of difficulties as encountered in proving completeness theorems for relevance logic in relational semantics. An exception is the relational constraint (C3), corresponding to Weakening being assumed. In a Boolean setting (for the logics of residuated Boolean algebras) (C3) can be replaced by the condition $\forall x,z,z'\in X\;(xR^{111}zz'\;\lra\;z'= x)$, but this is too strong a requirement for the semantics of substructural logics.

\begin{defn}\rm
\label{frame classes}
Taking the frame constraints of Table \ref{constraints} in consideration, we distinguish the frame classes $\mathbb{NFL, FL, BCI, BCW, BCK}$ by the obvious criteria, detailed below:
\begin{itemize}
\item[] $\mathbb{NFL}$) Unrestricted class of frames $\mathfrak{F}=(X,\upv,Y,R^{111})$
\item[] $\mathbb{FL}$)  Frames where the constraint (C1) holds
\item[] $\mathbb{BCI}$)  Frames where the constraints (C1) and (C2) hold
\item[] $\mathbb{BCK}$)  Frames where the constraints (C1), (C2) and (C3) hold
\item[] $\mathbb{BCW}$)  Frames where the constraints (C1), (C2) and (C4) hold
\end{itemize}
\end{defn}

\subsection{Canonical Extensions}
\label{canext section}
Canonical extensions for bounded lattices were introduced by Gehrke and Harding \cite{mai-harding}, extending to the non-distributive case the notion of a canonical extension of J\'{o}nsson and Tarski \cite{jt1,jt2}. A canonical lattice extension is defined as a pair $(\alpha,C)$ where $C$ is a complete lattice and $\alpha$ is an embedding of a lattice $\mathcal{L}$ into $C$ such that the following density and compactness requirements are satisfied
\begin{quote}
$\bullet$ (density) $\alpha[{\mathcal L}]$ is {\em dense} in $C$, where the latter means that every element of $C$ can be expressed both as a meet of joins and as a join of meets of elements in $\alpha[{\mathcal L}]$\\[0.5mm]
$\bullet$ (compactness) for any set $A$ of closed elements and any set  $B$ of open elements of $C$, $\bigwedge A\leq\bigvee B$ iff there exist finite subcollections $A'\subseteq A, B'\subseteq B$ such that $\bigwedge A'\leq\bigvee B'$
\end{quote}
where the {\em closed elements} of $C$ are defined in \cite{mai-harding} as the elements in the meet-closure of the representation map $\alpha$ and the {\em open elements} of $C$ are defined dually as the join-closure of the image of $\alpha$. Existence of canonical extensions for bounded lattices is proven in \cite{mai-harding} by demonstrating that the compactness and density requirements are satisfied in the representation due to Hartonas and Dunn \cite{sdl} (though it is to be noted that the lattice representation argument of \cite{sdl} is merely reproduced in \cite{mai-harding}, without attributing it to \cite{sdl}).

The canonical extension construction is functorial. Indeed, recall from \cite{mai-harding} that if $(\alpha,C)$ is a canonical extension of a bounded lattice $L$, and  $\type{K,O}$ are its sets of closed and open elements, the $\sigma$ and $\pi$-extensions $f_\sigma,f_\pi:{\mathcal L}_\sigma\lra{\mathcal L}_\sigma$ (where, following the notation of \cite{mai-harding}, ${\mathcal L}_\sigma$ designates the canonical extension of $\mathcal L$) of a unary monotone map $f:L\lra L$ are defined in \cite{mai-harding}, taking also into consideration Lemma 4.3 of \cite{mai-harding}, by setting, for $k\in\type{K}$, $o\in\type{O}$ and $u\in C$
\begin{eqnarray}
f_\sigma(k)=\bigwedge\{f(a)\midsp k\leq a\in L\} &\hskip1cm& f_\sigma(u)=\bigvee\{f_\sigma(k)\midsp\type{K}\ni k\leq u\}\label{sigma-extn}\\
f_\pi(o)=\bigvee \{f(a)\midsp L\ni a\leq o\} && f_\pi(u)=\bigwedge\{f_\pi(o)\midsp u\leq o\in\type{O}\}\label{pi-extn}
\end{eqnarray}
where in these definitions $\mathcal L$ is identified with its isomorphic image in $C$ and $a\in {\mathcal L}$ is then identified with its representation image.
Working concretely with the canonical extension of \cite{sdl}, the $\sigma$ extension $f_\sigma:{\mathcal L}_\sigma\lra{\mathcal L}_\sigma$ of a monotone map $f$ as in equation (\ref{sigma-extn}) and the dual $\sigma$-extension $f_\sigma^\partial:{\mathcal L}^\partial_\sigma\lra{\mathcal L}^\partial_\sigma$ (not used in \cite{mai-harding}) are defined by instantiating equation (\ref{sigma-extn}) in the concrete canonical extension of \cite{sdl} considered here by setting, for $x\in X$ and $y\in Y$ and where $x_e$ is a principal filter, $y_e$ a principal ideal and the closed elements are precisely the principal upper sets $\Gamma u$ (for each of $X$, $Y$).
\begin{tabbing}
$f_\sigma(\Gamma x)$\hskip2mm\==\hskip2mm\=$\bigwedge\{\alpha_X(fa)\midsp a\in{\mathcal L}, \Gamma x\leq\alpha_X(a)\}$\hskip2mm\==\hskip2mm\=$\bigwedge\{\Gamma x_{fa}\midsp\Gamma x\subseteq\Gamma x_a\}$\\
\>=\>$\bigwedge\{\Gamma x_{fa}\midsp a\in x\}$\>=\>$\Gamma(\bigvee\{x_{fa}\midsp a\in x\})$\\[1mm]
$f_\sigma^\partial(\Gamma y)$\>=\>$\bigwedge\{\alpha_Y(fa)\midsp a\in{\mathcal L}, \Gamma y\leq\alpha_Y(a)\}$\>=\>$\bigwedge\{\Gamma y_{fa}\midsp\Gamma y\subseteq\Gamma y_a\}$\\
\>=\>$\bigwedge\{\Gamma y_{fa}\midsp a\in y\}$\>=\>$\Gamma(\bigvee\{y_{fa}\midsp a\in y\})$
\end{tabbing}
From \cite{dloa} recall that if $f$ is an $n$-ary normal operator of distribution type $\delta=(i_1,\ldots,i_n;i_{n+1})$, where $i_j\in\{1,\partial\}$ (in other words $f:\mathcal{L}^{i_1}\times\cdots\times\mathcal{L}^{i_n}\lra\mathcal{L}^{i_{n+1}}$ is an additive operator), then  we defined filter/ideal operators $f^\sharp, f^\flat$ which, in the unary case are precisely the maps $f^\sharp(x)=\bigvee\{x_{fa}\midsp a\in x\}$ and $f^\flat(y)=\bigvee\{y_{fa}\midsp a\in y\}$ (for a filter $x$ and ideal $y$).
It can be then immediately observed that $f_\sigma(\Gamma x)=\Gamma(f^\sharp x)$, if $f$ distributes over joins, and that $f^\partial_\sigma(\Gamma y)=\Gamma(f^\flat y)$ if it distributes over meets (the cases of co-distribution being reducible to these two by considering $f$ as a map from $\mathcal{L}$ to its opposite lattice $\mathcal{L}^\partial$, or the other way around). The $\pi$-extension $f_\pi$ is obtained from the dual $\sigma$ extension $f^\partial_\sigma$, by composition with the Galois maps, $f_\pi(\lperp\{y\})=\lperp(f^\partial_\sigma(\lperp\{y\})\rperp)=\lperp(f^\partial_\sigma(\Gamma y))$, where the set $\{\lperp\{y\}\midsp y\in Y\}$ is the set of open elements of $\mathcal{L}_\sigma$ and $\lperp(\;), (\;)\rperp$ designates the Galois connection generated by $\upv$.

Besides serving as a means to construct the $\sigma$-extension of maps, as shown in \cite{sdl-exp,pnsds}, point operators are used in \cite{dloa} and in each of \cite{sdl-exp,pnsds} to construct relational structures dual to lattice expansions. Roughly, for each normal lattice operator $f:\mathcal{L}^{i_1}\times\cdots\times\mathcal{L}^{i_n}\lra\mathcal{L}^{i_{n+1}}$, where $i_j\in\{1,\partial\}$, a canonical relation $R_f$ is defined by setting $uR_fu_1\cdots u_n$ iff $f^\sharp(u_1,\ldots,u_n)\leq u$ and a dual relation $R^\partial_f$ is then appropriately derived. The relations $R_f,R^\partial_f$ are then used to define operators $F$ on stable sets and $F^\partial$ on co-stable sets that are shown in \cite{sdl-exp,pnsds} to be the canonical (dual) extension $f_\sigma, f^\partial_\sigma$ of $f$.

We conclude this section with considering the case where the represented lattice (expansion) is distributive, or a Boolean algebra.   For the approach based on and extending the Hartonas-Dunn lattice duality \cite{sdl}, the canonical frames for distributive and Boolean algebras are the obvious subframes of the canonical lattice frame. Indeed, let $\mathcal{L}$ be a lattice and assume it is a Boolean algebra. Let ${\mathcal G}(\filt(\mathcal{L}))$ be the set of Galois-stable subsets of the Hartonas-Dunn dual canonical frame $(\filt(\mathcal{L}),\upv,\idl(\mathcal{L}))$ (the canonical extension of $\mathcal{L}$) and let $\ufilt(\mathcal{L})\subseteq\filt(\mathcal{L})$ be its set of ultrafilters. It is not hard to see that the map $\eta:{\mathcal G}(\filt(\mathcal{L}))\lra \powerset(\ufilt(\mathcal{L}))$ defined by $\eta(A)=A\cap\ufilt(\mathcal{L})$ is an isomorphism of canonical extensions. Indeed, by uniqueness of canonical extensions \cite{mai-harding},  there is an isomorphism $\eta:{\mathcal G}(\filt(\mathcal{L}))\iso \powerset(\ufilt(\mathcal{L}))$, commuting with the embeddings $\alpha_A(a)=\{x\in\filt(\mathcal{L})\midsp a\in x\}$, \mbox{$h_A(a)=\{u\in\ufilt(\mathcal{L})\midsp a\in u\}$,} i.e. $\eta(\alpha_A(a))=h_A(a)$, for all $a\in\mathcal{L}$. Given that by join density of principal filters any filter $x$ is $x=\bigvee_{a\in x}x_a$, where we let $x_a=a\!\uparrow$ designate a principal filter and that, letting $\Gamma x=x\!\Uparrow$,  $\Gamma x=\Gamma(\bigvee_{a\in x}x_a)=\bigwedge_{a\in x}\Gamma x_a=\bigwedge_{a\in x}\alpha_A(a)$ and since $\eta$ preserves both meets and joins, being an isomorphism, we obtain
\begin{tabbing}
$\eta(\Gamma x)$\hskip3mm\==\hskip2mm\= $\eta(\bigwedge_{a\in x}\alpha_A(a))=\bigcap_{a\in x}\eta(\alpha_A(a))=\bigcap_{a\in x}h_A(a)$\\
\> =\> $\bigcap_{a\in x}\{u\in\ufilt(\mathcal{L})\midsp a\in u\}= \{u\in\ufilt(\mathcal{L})\midsp x\leq u\}$\\
\>=\> $\{u\in\ufilt(\mathcal{L})\midsp u\in\Gamma x\}=   \Gamma x\cap \ufilt(\mathcal{L})$
\end{tabbing}
where $\leq$ in `$x\leq u$' designates filter inclusion.
Given also join-density of the set of closed elements of ${\mathcal G}(\filt(\mathcal{L}))$ (see Lemma 3.1, \cite{discr}) we obtain that for any stable set $A\in{\mathcal G}(\filt(\mathcal{L}))$,
\begin{eqnarray*}
\eta(A)&=&\eta(\bigvee_{x\in A}\Gamma x)=\bigcup_{x\in A}\eta(\Gamma x)=\bigcup_{x\in A}\bigcap_{a\in x}\{u\in\ufilt(\mathcal{L})\midsp a\in u\}\\
&=&\bigcup_{x\in A}\{u\in\ufilt(\mathcal{L})\midsp x\leq u\}
\end{eqnarray*}
Stable sets are increasing, i.e. $x\in A$ implies $\Gamma x\subseteq A$ (see Lemma 3.1, \cite{discr}), hence $\eta(A)=A\cap\ufilt(\mathcal{L})$.

The following lemma will be useful in the sequel.
\begin{lemma}\rm
\label{eta morphism}
Let $\mathcal{A}$ be a Boolean algebra with a unary monotone operator $f:\mathcal{A}\lra\mathcal{A}$. Let also $\mathcal{A}_\sigma$ be the Hartonas-Dunn canonical extension $\mathcal{A}_\sigma=\mathcal{G}(\filt(\mathcal{A}))$ of $\mathcal{A}$ considered merely as a lattice and let $\mathcal{A}_\delta=\powerset(\ufilt(\mathcal{A}))$ be the standard construction of the canonical (ultrafilter) extension of the Boolean algebra $\mathcal{A}$, after Stone \cite{stone1}. Designate the canonical extension ($\sigma$-extension) of the map $f$ respectively by $f_\sigma$ and $f_\delta$. Then, the isomorphism $\eta:\mathcal{G}(\filt(\mathcal{A}))\iso \powerset(\ufilt(\mathcal{A}))$ is a homomorphism of lattice expansions, i.e. it commutes with the $\sigma$-extensions of $f$, $\eta f_\sigma=f_\delta\eta$.
\end{lemma}
\begin{proof}
For the purposes of this proof, let $X=\filt(\mathcal{A})$, $X^u=\ufilt(\mathcal{A})$,  $\Gamma x=\{z\in X\midsp x\leq z\}$, for a filter $x$ and where $\leq$ designates filter inclusion, and $\Gamma^ux=\{u\in X^u\midsp x\leq u\}$ and recall that it was argued above that $\eta(\Gamma x)=\Gamma^ux$. Letting $\alpha:\mathcal{A}\hookrightarrow\mathcal{A}_\sigma$ and $h:\mathcal{A}\hookrightarrow\mathcal{A}_\delta$ be the canonical embeddings of $\mathcal{A}$ it is not hard to see that the intersection closure of $\alpha[\mathcal{A}]$ is the collection of $\Gamma x$, with $x\in X$ and, similarly, the intersection closure of $h[\mathcal{A}]$ is the collection of the sets $\Gamma^ux$, with $x\in X$. In other words, the sets of the form $\Gamma x,\Gamma^ux$ are the closed (filter) elements of $\mathcal{A}_\sigma$ and $\mathcal{A}_\delta$, respectively. By  $\eta(\Gamma x)=\Gamma^ux$, the isomorphism of canonical extensions preserves closed elements (and similarly for open elements). Given the definition of the $\sigma$-extension of a monotone map in $\mathcal{A}_\sigma$ we obtain
\begin{equation}
\eta f_\sigma(\Gamma x)=\{u\in X^u\midsp \forall a\in\mathcal{A}\;(a\in x\;\lra\;f(a)\in u)\}
\label{eta-sigma}
\end{equation}
Similarly, considering the $\sigma$-extension of $f$ in $\mathcal{A}_\delta=\powerset(X^u)$ we obtain by definition
\begin{equation}
f_\delta(\eta(\Gamma x))=f_\delta(\Gamma^ux)=\{u\in X^u\midsp\forall a\in\mathcal{A}\;(\Gamma^ux\subseteq h(a)\;\lra\;f(a)\in u)\}
\label{sigma-eta}
\end{equation}
where $\Gamma^ux\subseteq h(a)$ iff $\eta(\Gamma x)\subseteq\eta\alpha(a)$ iff $\Gamma x\subseteq\alpha(a)$ iff $a\in x$. Hence we obtain that $\eta f_\sigma(\Gamma x)=f_\delta(\eta(\Gamma x))$, as claimed.
\end{proof}
Using the fact pointed out in \cite{mai-harding} that $(\mathcal{L}\times\mathcal{M})_\sigma\iso\mathcal{L}_\sigma\times\mathcal{M}_\sigma$ and $(\mathcal{L}^\partial)_\sigma\iso(\mathcal{L}_\sigma)^\partial$, Lemma \ref{eta morphism} extends readily to $n$-ary maps that are either monotone, or antitone in their argument places.

\section{Sorted Frames with Relations and their Logics}
\label{rel-section}
The extensions of {\bf ML}$_2$ that we consider arise as the logics of sorted frames with additional relations. Cases of interest include temporal formal contexts $((X,\prec),\upv,Y)$, where $\prec\;\subseteq X\times X$ is a relation of temporal succession of (instances of) formal objects, of use in Temporal FCA \cite{wolff}, or contexts with an indiscernibility relation $\approx$ on formal objects, $((X,\approx),\upv,Y)$, arising in Rough FCA \cite{kent}. Frames of this type correspond to extensions of substructural logics with unary modal operators, which we do not discuss in this article.

We focus here on structures $(X,\upv,Y,R^{111})$, where $R$ has the sorting type indicated by the superscript (with $1$ pointing to $X$ and $\partial$ to $Y$), i.e. $R\subseteq X\times (X\times X)$. With a spatial intuition, as in \cite{yde-points-lines}, $R$ can be understood, for example, as a colinearity relation for points in $X$, $xRx'x''$ iff $\exists y\in Y\;\{x,x',x''\}\upv y$, where $\upv$ is the incidence relation of the frame. We make no further assumptions about the accessibility relation $R^{111}$, such as the `compatibility' assumption made in \cite{mai-gen,mai-grishin}, because we wish to ensure modal definability of our frame classes. It is for the same reason that we do not assume that frames are separated and/or reduced.

\subsection{Sorted, Residuated Modal Logics}
\label{operators on sets section}
Given a frame $(X,\upv,Y,R^{111})$, the relation $R^{111}$ generates a classical binary additive operator on $\powerset(X)$, defined in the standard way by
\begin{eqnarray}
U\mbox{$\bigodot$}U'&=&\{x\in X\midsp \exists u,u'\in X\;(xR^{111}uu'\;\wedge\;u\in U\;\wedge\;u'\in U')\}
\label{bigodot}\\
&=& \bigcup_{u\in U, u'\in U'}R^{111}uu'\nonumber
\end{eqnarray}
It follows from the definition that $\bigodot$ distributes over arbitrary unions in each argument place and we let $\tilde{\La},\tilde{\Ra}$ be its residuals in the algebra of subsets of $X$, i.e. $U\subseteq W\tilde{\La} U'$ iff $U\bigodot U'\subseteq W$ iff $U'\subseteq U\tilde{\Ra} W$. Hence $(\powerset(X),\tilde{\La},\bigodot,\tilde{\Ra})$ is a residuated Boolean algebra.
\begin{lemma}\rm
\label{modal residuals lemma}
The residuals $\tilde{\La},\tilde{\Ra}$ are defined by
\begin{itemize}
\item $U\tilde{\Ra} W=\{x\in X\midsp\forall z,z'\in X\;(z\in U\;\wedge\;z'R^{111}zx\;\lra\;z'\in W)\}$
\item $W\tilde{\La} U=\{x\in X\midsp \forall z,z'\in X\;(z\in U\;\wedge\;z'R^{111}xz\;\lra\;z'\in W)\}$
\end{itemize}
\end{lemma}
\begin{proof}
The proof is by the following calculation
\begin{tabbing}
$U\bigodot U'\subseteq W$\hskip2mm\= iff\hskip2mm\= $\forall z'\in X\;(\exists z,x\in X\;(z'R^{111}zx\;\wedge\;z\in U\;\wedge\;x\in U')\lra z'\in W)$\\
\>iff\> $\forall x,z,z'\in X\;(x\in U'\;\lra\;(z\in U\;\wedge\;z'R^{111}zx\;\lra\;z'\in W))$\\
\>iff\> $U'\subseteq \{x\in X\midsp \forall z,z'\in X\;(z\in U\;\wedge\;z'R^{111}zx\;\lra\;z'\in W)))\}$
\end{tabbing}
and since $U\bigodot U'\subseteq W$ iff $U'\subseteq U\tilde{\Ra}W$ the claim for $\tilde{\Ra}$ is true. Similarly,
$U\bigodot U'\subseteq W$ iff $U\subseteq\{z\in X\midsp\forall x,z'\in X\;(x\in U'\;\wedge\;z'R^{111}zx\;\lra\;z'\in W)\}$, then just change bound variables.
\end{proof}

Properties of the operator $\bigodot$ correspond to the relational constraints of Table \ref{constraints}.

\begin{lemma}\rm
\label{C and odot}
The following hold
\begin{enumerate}
  \item (C1) holds iff $\bigodot$ is {\em associative}
  \item (C2) holds iff $\bigodot$ is {\em commutative} (in which case $\tilde{\La}$ coincides with $\tilde{\Ra}$)
  \item If (C3) holds then $\bigodot$ is {\em thinning} on the family of $\preceq$-increasing sets, i.e. the inclusion $U\bigodot W\subseteq W$ holds for any $\preceq$-increasing $U,W\subseteq X$
  \item (C4) holds iff $\bigodot$ is {\em contractive}, i.e. the inclusion $U\cap W\subseteq U\bigodot W$ holds for any $U,W\subseteq X$
  \item If both (C2) and (C3)  hold, then $U\bigodot W\subseteq U\cap W$  for any $U,W\subseteq X$
  \item If (C2)-(C4) all hold, then (C1) also holds since $\bigodot$ is trivialized to the set intersection operation
\end{enumerate}
\end{lemma}
\begin{proof}
Assuming (C1), let $z\in(U\bigodot V)\bigodot W$,  $x\in U\bigodot V$ and $u\in U, v\in V, w\in W$ so that $xR^{111}uv$ and $zR^{111}xw$, so that $\exists x(xR^{111}uv\;\wedge\; zR^{111}xw)$ holds. By (C1), let $x'$ be such that $x'R^{111}vw$ and $zR^{111}ux'$, so that $x'\in V\bigodot W$ and $z\in U\bigodot(V\bigodot W)$, hence $(U\bigodot V)\bigodot W\subseteq U\bigodot(V\bigodot W)$. The other direction of the inclusion is obtained similarly, going from (2) to (1). Conversely, if $\bigodot$ is associative, then from the identity $(U\bigodot V)\bigodot W= U\bigodot(V\bigodot W)$ the equivalence of conditions (1) and (2) is obtained in a straightforward manner, as it is merely a restatement of the special case of association $(\{u\}\bigodot\{v\})\bigodot\{w\}=\{u\}\bigodot(\{v\}\bigodot\{w\})$, for any $u,v,w\in X$.

The case of (C2) and commutativity is immediate.

Assuming (C3), let $x\in U\bigodot V$, where $U,V$ are $\preceq$-increasing. Let $u\in U, v\in V$ so that $xR^{111}uv$. By (C3), $v\preceq x$ and since $v\in V$, we obtain $x\in V$ by $\preceq$-increasingness of $V$.

If $x\in U\cap V$ and (C4) holds, then by $xR^{111}xx$ it follows $x\in U\bigodot V$. Conversely, if $\bigodot$ is contractive, choose $U=\{x\}=V$ to get $\{x\}\subseteq\{x\}\bigodot\{x\}$ from which $xR^{111}xx$ follows.
\end{proof}

The sorted modal language of a frame $\mathfrak{F}=(X,\upv,Y,R^{111})$, where we let $I$ be the complement of $\upv$, is the language
\begin{eqnarray*}
\alpha &:=& P_i\;(i\in\mathbb{N})\midsp\neg \alpha\midsp \alpha\wedge \alpha\midsp \lbbox \beta\midsp \alpha\odot \alpha\midsp \alpha\lfspoon \alpha\midsp \alpha\rfspoon \alpha \\
\beta &:=& Q_i\;(i\in\mathbb{N})\midsp\neg \beta\midsp \beta\wedge \beta\midsp \largesquare \alpha
\end{eqnarray*}
The language contains a copy of the language of {\bf CPL} (in the second sort), a copy of the language of residuated Boolean algebras (in the first sort), synthesized by the sorted modalities on which a residuation axiom will be imposed.

The language is interpreted in models $\mathfrak{M}=(\mathfrak{F},\imath)$, where $\imath(P_i)\subseteq X$, $\imath(Q_i)\subseteq Y$ and the interpretation is extended to all sentences by adding to the clauses for  {\bf CPL} (for each sort) satisfaction clauses for the additional  constructs, as shown in Table \ref{sat modal}.

\begin{table}[!htbp]
\caption{Satisfaction clauses for {\bf sub.ML}$_2$}
\label{sat modal}
\begin{tabbing}
\hskip3mm\=$x\models\alpha\odot\alpha'$\hskip5mm\= iff\hskip2mm\= $\exists z,z'\in X\;(xR^{111}zz'\;\wedge\; z\models\alpha\;\wedge\;z'\models\alpha')$\\
\>$x\models \alpha\rfspoon\alpha'$\>iff\> $\forall z,z'\in X\;(z\models\alpha\;\wedge\;z'R^{111}zx\;\lra\;z'\models\alpha')$ \\
\>$x\models\alpha'\lfspoon\alpha$ \>iff\> $\forall z,z'\in X\;(z\models\alpha\;\wedge\;z'R^{111}xz\;\lra\;z'\models\alpha')$\\
\>$x\models\lbbox \beta$\> iff\>  $\forall y\in Y\;(xIy\;\lra\;y\models \beta)$\\
\>$y\models\largesquare \alpha$\>iff\>$\forall x\in X\;(xIy\;\lra\;x\models \alpha)$
\end{tabbing}
\end{table}

It follows directly from definitions that $\val{\alpha\odot\alpha'}=\{x\in X\midsp x\models\alpha\circ\alpha'\}=\val{\alpha}\bigodot\val{\alpha'}$ and similarly $\val{\alpha\rfspoon\alpha'}=\val{\alpha}\tilde{\Ra}\val{\alpha'}$, $\val{\alpha'\lfspoon\alpha}=\val{\alpha'}\tilde{\La}\val{\alpha}$.

The two-sorted logic $\Lambda_2(\mathbb{F})$ of a class $\mathbb{F}$ of frames $\mathfrak{F}=(X,\upv,Y,R^{111})$ is a pair of sets of sentences $\Lambda_2(\mathbb{F})=(\type{A},\type{B})$ such that for every $\alpha\in\type{A}$ and any model $\mathfrak{M}=(\mathfrak{F},\imath)$ over any frame $\mathfrak{F}=(X,\upv,Y,R^{111})$ in the frame class $\mathbb{F}$, $\val{\alpha}_\imath=X$ and similarly for $\beta\in\type{B}$, $\yvval{\beta}_\imath=Y$.
The hierarchy of the corresponding sorted, residuated modal logics is displayed in Figure \ref{modal-hierarchy}, where the modal companion logic of a substructural logic {\bf SUB} is designated by {\bf sub.ML}$_2$.

\begin{figure}[!htbp]
\caption{Hierarchy of Modal Companion Logics}
\label{modal-hierarchy}
{\small
\[
\xymatrix{
{\bf bck.ML}_2 \mbox{ (C1)(C2)(C3)-frames} & {\bf bcw.ML}_2 \mbox{ (C1)(C2)(C4)-frames}  \\
& {\bf bci.ML}_2\ar@{-}[ul]\ar@{-}[u]\mbox{ (C1)(C2)-frames}\\
& {\bf fl.ML}_2\ar@{-}[u] \mbox{ (C1)-frames}\\
&{\bf nfl.ML}_2\ar@{-}[u]
}
\]
}
\end{figure}
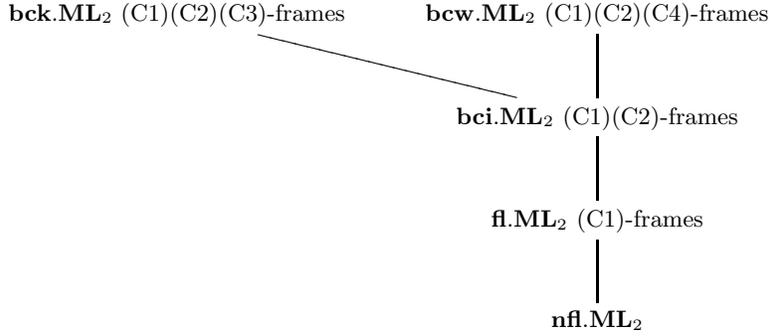

Note that for the systems including and above {\bf bci.ML}$_2$, the signature of the language may be restricted by dropping the (superfluous) $\lfspoon$.\\

\paragraph{Sorted Modal Logic: The Basic System ML$_2$}\mbox{}\\
The language of the minimal two-sorted, residuated modal logic ({\bf ML}$_2$) is generated by the following scheme
\begin{eqnarray*}
 \alpha &:=& P_i\;(i\in\mathbb{N})\midsp\neg \alpha\midsp \alpha\wedge \alpha\midsp \lbbox \beta \\
 \beta &:=& Q_i\;(i\in\mathbb{N})\midsp\neg \beta\midsp \beta\wedge \beta\midsp \largesquare \alpha
\end{eqnarray*}
Further logical connectives and constants are defined as usual in classical normal modal logic, i.e. $\alpha\vee \alpha'=\neg(\neg \alpha\wedge\neg \alpha')$, $\top=P_0\vee\neg P_0$, $\bot=P_0\wedge\neg P_0$, $\alpha\ra \alpha'=\neg \alpha\vee \alpha'$ and similarly for sentences of the second sort. In addition, diamond operators are defined by $\largediamond \alpha=\neg\largesquare\neg \alpha$ and $\lbdiamond \beta=\neg\lbbox\neg \beta$, where the two occurrences of negation are of different sort. Note that there are distinct constants $\top_1,\top_2$ (and similarly for $\bot_1,\bot_2$), corresponding to each sort. We avoid subscripts by letting $\sub{\top}{1}=\top,\sub{\bot}{1}=\bot$ and $\sub{\top}{2}={\tt t},\sub{\bot}{2}={\tt f}$.

Models $\mathfrak{M}=(\mathfrak{F},\imath)$ for  {\bf ML}$_2$  on two-sorted frames $\mathfrak{F}=(X,I,Y), I\subseteq X\times Y,$ (always assumed to satisfy condition \eqref{D-cond})
are defined as detailed for the general case of {\bf sub.ML}$_2$.

Let {\bf ML}$^\bullet_2$ designate the set of sentences of the first sort of the form $\lbbox\beta$, for some sentence $\beta$ of the second sort. Similarly, let {\bf ML}$^\circ_2$ designate the set of sentences of second sort of the form $\largesquare\alpha$, for some sentence $\alpha$ of the first sort. If $A$ is an {\bf ML}$^\bullet_2$ sentence and $B$ is an {\bf ML}$^\circ_2$ one, then clearly, $\val{A}=\val{\lbbox\largediamond A}=\lbbox\largediamond \val{A}$ and $\yvval{B}=\yvval{\largesquare\lbdiamond B}=\largesquare\lbdiamond\yvval{B}$.

In accordance with our decision for the substructural logics proof systems we adopt a symmetric consequence system, directly encoding the algebraic specification of the logic. Thus,
sequents for the logic {\bf ML}$_2$ are of the form $\alpha\proves\alpha'$, where $\alpha,\alpha'$ are  sentences of the first sort, as well as of the form $\beta\vproves\beta'$, where $\beta,\beta'$ are sentences of the second sort. The proof system consists of axioms and rules for sequents of the first type and of the second type, which are copies of the rules for classical propositional logic and can be found in any standard reference.  Furthermore, bridge rules are included, shown in Table \ref{bridge rules for ML2}, to relate the two subsystems.
\begin{table}[!htbp]
\caption{Residuation Axioms, D-Axioms and Bridge Rules}
\label{bridge rules for ML2}
\begin{tabbing}
$\alpha\proves \lbbox\largediamond \alpha$\hskip3mm\=
$\largediamond\lbbox \beta\vproves \beta$
\hskip3mm \=$\lbbox \beta\proves\lbdiamond \beta$\hskip3mm\=$\largesquare \alpha\vproves\largediamond \alpha$
\hskip3mm \= $\top\proves\lbbox\top$ \hskip3mm \= $\type{t}\proves\largesquare\type{t}$\\[1mm]
$\infrule{\alpha\proves \alpha'}{\largesquare\alpha\vproves\largesquare \alpha'}$ \>
$\infrule{\beta\vproves \beta'}{\lbbox\beta\proves\lbbox \beta'}$ \> $\largesquare\alpha\wedge\largesquare\alpha'\vproves\largesquare(\alpha\wedge\alpha')$ \>\> $\lbbox\beta\wedge\lbbox\beta'\proves\lbbox(\beta\wedge\beta')$
\end{tabbing}
\end{table}

For semantic entailment we write $\alpha\models\alpha'$ and $\beta\models\beta'$ (we do not change notation between the two sorts), with a standard, classical understanding. As usual, we write $\proves\alpha$ stand for $\top\proves\alpha$, as well as $\models\alpha$ for $\top\models\alpha$ and similarly for $\proves\beta$ and $\models\beta$. It is immediate that {\bf ML}$_2$ is sound in the class of two-sorted frames satisfying condition \eqref{D-cond}.

The Lindenbaum-Tarski algebra of the logic is naturally a two-sorted algebra $(\mathcal{A}, \mathcal{B},\Diamond,\bbox)$ where both ${\mathcal A,B}$ are Boolean algebras, $\Diamond:\mathcal{A}\rightleftarrows\mathcal{B}:\bbox$ are normal modal operators satisfying the D-axiom $\bbox b\leq\blackdiamond b$ and $\Box a\leq\Diamond a$ and, in addition, for any $a\in\mathcal{A}, b\in\mathcal{B}$, $\Diamond a\leq b$ iff $a\leq\bbox b$ and then also $\blackdiamond b\leq a$ iff $b\leq\Box a$ where boxes and diamonds are interdefinable using classical complementation, as usual.

Soundness is immediate and completeness of {\bf ML}$_2$ is proven along the lines of the completeness proof for  normal modal logics, except that we consider both the set $X=\ufilt(\mathcal{A})$ of ultrafilters of $\mathcal{A}$ (consisting of equivalence classes of sentences of the first sort), as well as the set $Y=\ufilt(\mathcal{B})$ of ultrafilters of $\mathcal{B}$ (consisting of equivalence classes of sentences of the second sort). The canonical relation $I\subseteq X\times Y$ is defined along the standard lines: $uIv$ iff for all $b\in\mathcal{B}$, if $\bbox b\in u$, then $b\in v$. This can be shown to be equivalent to the definition \mbox{$\forall a\;(a\in u\;\lra\;\blackdiamond a\in v)$,} as usual. Furthermore, by residuation, the defining condition for the accessibility relation $I$ can be stated analogously and equivalently using $\Box$ (or $\Diamond$).
We state the completeness result below; details are left to the interested reader.
\begin{prop}\rm
{\bf ML}$_2$ is sound and complete in two-sorted frames satisfying conditions \eqref{D-cond}.\telos
\end{prop}

The proof systems of the modal companion logics {\bf sub.ML}$_2$, extending {\bf ML}$_2$, are specified as follows.\\

\paragraph{nfl.ML$_2$} Monotonicity and residuation axioms for $\lfspoon,\odot,\rfspoon$ are added\\

\begin{tabular}{llllllll}
$\infrule{\alpha\proves\alpha_1\hskip5mm\alpha'\proves\alpha'_1}{\alpha\odot\alpha'\proves\alpha_1\odot\alpha'_1}$ &
\\[1mm]
$\alpha\odot(\alpha\rfspoon\alpha')\proves\alpha'$ & $\alpha'\proves\alpha\rfspoon(\alpha\odot\alpha')$ & $\infrule{\alpha\proves\alpha_1\hskip5mm\alpha'_1\proves\alpha'}{\alpha_1\rfspoon\alpha'_1\proves\alpha\rfspoon\alpha'}$\\[1mm]
$\alpha'\proves(\alpha\odot\alpha')\lfspoon\alpha$ & $(\alpha'\lfspoon\alpha)\odot\alpha\proves\alpha'$
&$\infrule{\alpha_1\proves\alpha\hskip5mm\alpha'\proves\alpha_1'}{\alpha\lfspoon\alpha'\proves\alpha_1\lfspoon\alpha_1'}$
\end{tabular}\\

\paragraph{fl.ML$_2$\ =\ {\bf nfl.ML}$_2$ + Associativity axioms}\mbox{}\\

$\alpha_1\odot(\alpha_2\odot\alpha_3)\proves (\alpha_1\odot\alpha_2)\odot\alpha_3$ and $(\alpha_1\odot\alpha_2)\odot\alpha_3\proves \alpha_1\odot(\alpha_2\odot\alpha_3)$ \\

\paragraph{bci.ML$_2$\ =\ {\bf fl.ML}$_2$ + Commutativity axiom}\mbox{}\\

$\alpha\odot\alpha'\proves\alpha'\odot\alpha$ \\

\paragraph{bcw.ML$_2$\ =\ {\bf bci.ML}$_2$ + Contraction axiom}\mbox{}\\

$\alpha\wedge\alpha'\proves\alpha\odot\alpha'$\\

\paragraph{bck.ML$_2$\ =\ {\bf bci.ML}$_2$ + Controlled  weakening axiom}\mbox{}\\

$\lbbox\beta\odot\lbbox\beta'\proves\lbbox\beta'$ \\

It follows from the axiomatization that the Lindenbaum algebra of {\bf nfl.ML}$_2$ is a sorted, residuated modal algebra $\Diamond:\mathcal{A}\leftrightarrows\mathcal{B}:\bbox$, where in addition $(\mathcal{A},\lfspoon,\circ,\rfspoon)$ is a residuated Boolean algebra. Adding associativity, commutativity and contraction axioms, the operator $\circ$ of the algebra $(\mathcal{A},\lfspoon,\circ,\rfspoon)$ is associative, commutative, contractive, respectively, and, with commutativity, $\lfspoon=\rfspoon$.
Note that contractiveness of $\odot$ can be assumed, yielding {\bf bcw.ML}$_2$, but not weakening, which can only be introduced in a controlled way (reminiscent of Linear Logic's  ! operator). Recall that satisfaction clauses (on top of the usual clauses for {\bf CPL}, for each sort of sentences) have been specified in \mbox{Table \ref{sat modal}.}

\begin{thm}[sub.ML$_2$ Soundness]\rm
\label{subML2 soundness}
{\bf sub.ML}$_2$ is sound in $\mathbb{SUB}$-frames.
\end{thm}
\begin{proof}
The proof was essentially given in Lemma \ref{C and odot}. For soundness of the {\bf bck.ML}$_2$ axiom $\lbbox\beta\odot\lbbox\beta'\proves\lbbox\beta'$ it only needs to be observed that, given a frame $(X,\upv,Y,R^{111})$, for any subset $V\subseteq Y$, the set $\lbbox V\subseteq X$ is $\preceq$-increasing, since it is a Galois stable set $\lbbox V=\lbbox\largediamond\lbbox V$ (by residuation).
\end{proof}

\subsection{Formal Concept Lattice Logics}
\label{stable ops section}
Let $\mathfrak{F}=(X,\upv,Y,R^{111})$ be a frame and let $\bigovert$ be the closure of the restriction of $\bigodot$ on the set $\gpsi$ of Galois stable subsets of $X$, i.e.   $A\bigovert C=\lbbox\largediamond(A\bigodot C)$.  The Galois dual operator $\bigovert^{\!\partial}$ is defined on co-stable sets $B=A\rperp, D=C\rperp$ by
\begin{equation}
\mbox{$B\bigovert^{\!\partial}D=(\lperp B\bigovert\lperp D)\rperp=(A\bigovert C)\rperp=\bigcap_{u\in A, u'\in C}R^{\partial 11}uu'$}\label{circ dual}
\end{equation}
where we let $R^{\partial 11}$ be defined by $R^{\partial 11}uu'=(R^{111}uu')\rperp$, i.e. for any $y\in Y$ it holds that $yR^{\partial 11}uu'$ iff  $\forall x\in X\;(xR^{111}uu'\lra x\upv y)$. Hence we have $B\bigovert^{\!\partial}D=\bigcap_{u\in A, u'\in C}R^{\partial 11}uu'$ and it then follows that $\bigovert^\partial$ distributes over arbitrary intersections of co-stable sets in each argument place. Therefore, by duality, $\bigovert$ distributes over arbitrary joins of Galois stable sets and hence it is residuated in $\gpsi$, with residuals $\La,\Ra$, i.e. $C\subseteq F\La A$ iff $A\bigovert C\subseteq F$ iff $F\subseteq A\Ra C$.
\begin{lemma}\rm
\label{sub la ra}
The residuals $\La,\Ra$ of $\bigovert$ are the restrictions of the residuals $\tilde{\La},\tilde{\Ra}$ of $\bigodot$ on stable sets. They are equivalently defined by
\begin{tabbing}
$A\Ra C$\ \==\ \=$\{x\in X\midsp \forall z\in X\;\forall v\in Y\;(z\in A\;\wedge\;C\upv v\;\lra\; vR^{\partial 11}zx))\}$\\
$C\La A$\>=\> $\{x\in X\midsp \forall z\in X\;\forall v\in Y\;(z\in A\;\wedge\;C\upv v\;\lra\; vR^{\partial 11}xz))\}$
\end{tabbing}
\end{lemma}
\begin{proof}
Note that for a stable set $A=\lbbox\largediamond A$ and any set $U$ the inclusion $\lbbox\largediamond U\subseteq A$ obtains iff the inclusion $U\subseteq A$ does. In particular, $A\bigovert C\subseteq F$ iff $A\bigodot C\subseteq F$, from which the first claim of the corollary follows. For the second claim we do only the case for $\Ra$ (repeating the computation from \cite{discres}).
\begin{tabbing}
\hskip5mm\= $A\bigovert F\subseteq C$ \\
iff\> $A\bigodot F\subseteq C$ \\
iff\> $\forall z'\in X\;(z'\in A\bigodot F\;\lra\; z'\in C)$\\
iff\> $\forall z'\;(\exists x\in X\;\exists z\in X\;(z'R^{111}zx\;\wedge\;z\in A\;\wedge\;x\in F)\;\lra\;z'\in C)$\\
iff\> $\forall x,z,z'\in X\;((z'R^{111}zx\;\wedge\;z\in A\;\wedge\;x\in F)\;\lra\;z'\in C)$\\
iff\> $\forall x,z,z'\in X\;((z\in A\;\wedge\;x\in F)\;\lra\;(z'R^{111}zx\;\lra\;z'\in C))$\\
iff\> $\forall x,z\in X\;((z\in A\;\wedge\;x\in F)\;\lra\;\forall z'\in X\;(z'R^{111}zx\;\lra\;z'\in C))$\\
iff\> $\forall x,z\in X\;((z\in A\;\wedge\;x\in F)\;\lra\;(R^{111}zx\;\subseteq\; C))$\\
iff\> $\forall x,z\in X\;((z\in A\;\wedge\;x\in F)\;\lra\;(C\rperp\subseteq(R^{111}zx)\rperp))$\\
iff\> $\forall x,z\in X\;((z\in A\;\wedge\;x\in F)\;\lra\;(C\rperp\subseteq R^{\partial 11 }zx))$ \\
iff\> $\forall x,z\in X\;((z\in A\;\wedge\;x\in F)\;\lra\;\forall v\in Y\;(C\upv v\;\lra\; vR^{\partial 11  }zx))$\\
iff\> $\forall x,z\in X\;\forall v\in Y\;((z\in A\;\wedge\;x\in F)\;\lra\;(C\upv v\;\lra\; vR^{\partial 11 }zx))$\\
iff\> $\forall x,z\in X\;\forall v\in Y\;(x\in F\;\lra\;(z\in A\lra (C\upv v\;\lra\; vR^{\partial 11  }zx)))$\\
iff\> $\forall x,z\in X\;\forall v\in Y\;(x\in F\;\lra\;(z\in A\;\wedge\;C\upv v\;\lra\; vR^{\partial 11 }zx))$\\
iff\> $\forall x\in X\;(x\in F\;\lra\; \forall z\in X\;\forall v\in Y\;(z\in A\;\wedge\;C\upv v\;\lra\; vR^{\partial 11  }zx))$\\
iff\> $F\subseteq \{x\in X\midsp \forall z\in X\;\forall v\in Y\;(z\in A\;\wedge\;C\upv v\;\lra\; vR^{\partial 11 }zx))\}$\\
iff\> $F\subseteq A\Ra C$
\end{tabbing}
The case for $\La$ can be obtained from the above by appropriately modifying the last few lines, then changing bound variables to obtain the definition of $\La$ as stated in the lemma.
\end{proof}

\begin{lemma}\rm
\label{odot2overt}
If $\bigodot$ is associative, commutative, thinning, or contractive, then so is $\bigovert$. Consequently, the obvious correlation obtains between the frame constraints of Table \ref{constraints} and properties of the $\bigovert$ operator on the complete lattice $\gpsi$ of Galois stable subsets of $X$ (which is isomorphic to the formal concept lattice of the frame).
\end{lemma}
\begin{proof}
We first verify that associativity of $\bigodot$ implies associativity of $\bigovert$. Recall first that if $A$ is a stable set, then $\lbbox\largediamond U\subseteq A$ iff $U\subseteq A$, for any set $U$. This implies in particular that $F\subseteq A\Ra C$ (iff $A\bigovert F\subseteq C$) iff $A\bigodot F\subseteq C$, for stable sets $A,F,C$. We then obtain
\begin{tabbing}
$A\bigovert(C\bigovert E)\subseteq (A\bigovert C)\bigovert E$\hskip2mm\=iff\hskip2mm\=
$A\bigodot(C\bigovert E)\subseteq (A\bigovert C)\bigovert E$\\
\>iff\> $C\bigovert E\subseteq A\Ra((A\bigovert C)\bigovert E)$\\
\>iff\> $C\bigodot E\subseteq A\Ra((A\bigovert C)\bigovert E)$\\
\>iff\> $A\bigodot(C\bigodot E)\subseteq (A\bigovert C)\bigovert E$\\
Similarly,\\
$(A\bigovert C)\bigovert E\subseteq A\bigovert(C\bigovert E)$ \>iff\> $(A\bigodot C)\bigodot E\subseteq A\bigovert(C\bigovert E)$
\end{tabbing}
Hence, given associativity of $\bigodot$ and given that we have, by monotonicity, $A\bigodot(C\bigodot E)\subseteq A\bigovert(C\bigovert E)$ and, similarly, $(A\bigodot C)\bigodot E\subseteq (A\bigovert C)\bigovert E$, it follows that $\bigovert$ is associative. By duality, $B\bigovert^{\!\partial}D=(\lperp B\bigovert\lperp D)\rperp$ and $A\bigovert C=\lperp(A\rperp\bigovert^{\!\partial} C\rperp)$, this implies associativity of $\bigovert^{\!\partial}$.   For the reader's reassurance we display the calculation, given formal concepts $(A,B), (C,D), (E,F)$ (i.e. $B=A\rperp, D=C\rperp, F=E\rperp$).
\begin{tabbing}
$(B\bigovert^{\!\partial}D)\bigovert^{\!\partial} F$ \hskip3mm\= =\hskip2mm\= $(A\rperp\bigovert^{\!\partial} C\rperp)\bigovert^{\!\partial}E\rperp$ \hskip3mm\= =\hskip2mm\=
$(A\bigovert C)\rperp\bigovert^{\!\partial} E\rperp$\\
\>=\> $((A\bigovert C)\bigovert E)\rperp$ \>=\> $(A\bigovert(C\bigovert E))\rperp$\\
\>=\> $A\rperp\bigovert^{\!\partial}(C\bigovert E)\rperp$ \>=\> $A\rperp\bigovert^{\!\partial}(C\rperp \bigovert^{\!\partial} E\rperp)$\\
\>=\> $B\bigovert^{\!\partial}(D\bigovert^{\!\partial}F)$
\end{tabbing}
It is clear that $\bigovert$ is commutative iff $\bigodot$ is.

For the thinning property, assuming $\bigodot$ is thinning on $\preceq$-increasing sets then in particular for stable sets $A,C$ we have $A\bigodot C\subseteq C$, hence $\lbbox\largediamond(A\bigodot C)\subseteq C$, given stability of $C$, which means by definition that $A\bigovert C\subseteq C$.

For contractiveness, assume $U\cap U'\subseteq U\bigodot U'$ for any sets $U,U'\subseteq X$. In particular for stable sets $A\cap C\subseteq A\bigodot C$ and, taking closure, $A\cap C\subseteq A\bigovert C$.
\end{proof}

Language and semantics for formal concept logics are those for substructural logics, presented in Section \ref{sub review}. As noted in Section \ref{sub review}, it is sufficient for the purposes of this article to consider proof systems as symmetric consequence systems $\varphi\proves\psi$, directly encoding the algebraic specification of the logic. The relational semantics of the substructural systems has been specified in Table \ref{sat}. Algebraic soundness and completeness are well-established in the literature on substructural logics and, thereby, completeness in relational semantics can be established by a representation argument. For soundness, we verify that the algebra of stable subsets of the set $X$ of a frame $\mathfrak{F}=(X,\upv,Y,R^{111})$ is in the equational class of the Lindenbaum-Tarski algebra of the logic.

\begin{thm}[Soundness for SUB]\rm
\label{sub soundness}
Every substructural system {\bf SUB} of Figure \ref{sub-hierarchy} is sound in the corresponding frame class $\mathbb{SUB}$, as defined in Definition \ref{frame classes}.
\end{thm}
\begin{proof}
Let $\mathfrak{F}=(X,\upv,Y,R^{111})$ be a frame and $\gpsi$ its complete lattice of stable sets. In Sections \ref{operators on sets section} and \ref{stable ops section} we have shown that $\gpsi$ is equipped with an operator $\bigovert$, residuated with implication operators $\La,\Ra$. This suffices for soundness of {\bf NFL} in the class $\mathbb{NFL}$.

For the associative Lambek Calculus {\bf FL}, the corresponding frame class $\mathbb{FL}$ assumes the constraint (C1) of Table \ref{constraints}. In Lemma \ref{C and odot} it was shown that (C1) is equivalent to associativity of $\bigodot$ and in Lemma \ref{odot2overt} we verified that associativity of $\bigodot$ implies associativity of $\bigovert$, which was defined on stable sets by setting $A\bigovert C=\lbbox\largediamond(A\bigodot C)$. Furthermore, it was shown in Section \ref{stable ops section} that $\bigovert$ is residuated in $\gpsi$ (and we in fact verified in Lemma \ref{sub la ra} that its residuals $\La,\Ra$ are the respective restrictions of the residuals $\tilde{\La},\tilde{\Ra}$ of $\bigodot$ in the algebra of all subsets of $X$). This establishes soundness of {\bf FL} in its homonym frame class $\mathbb{FL}$.

For {\bf BCI}, Lemma \ref{C and odot} verified the equivalence between the constraint (C2) and commutativity of $\bigodot$, from which commutativity of $\bigovert$ immediately follows, as observed in Lemma \ref{odot2overt}. Therefore, {\bf BCI} is sound in its corresponding frame class $\mathbb{BCI}$.

For {\bf BCW}, assuming contraction, it was verified in Lemma \ref{C and odot} that the relational constraint (C4) in the frame class $\mathbb{BCW}$ is equivalent to contractiveness of $\bigodot$, which immediately implies contractiveness of $\bigovert$, as observed in Lemma \ref{odot2overt}, i.e. the inequation $a\wedge b\leq a\circ b$ holds in $\gpsi$. In other words, {\bf BCW} is sound in the class $\mathbb{BCW}$.

For {\bf BCK}, assuming weakening, it was verified in Lemma \ref{C and odot} that if the relational constraint (C3) (i.e. $\forall x,z,z'\in X\;(xR^{111}zz'\lra z'\preceq x)$) holds, then the inclusion $U\bigodot V\subseteq V$ holds for $\preceq$-increasing sets, and since every stable set is $\preceq$-increasing it follows that $A\bigodot C\subseteq C$, from which we obtain $A\bigovert C=\lbbox\largediamond(A\bigodot C)\subseteq C$, given that we assume $C$ to be stable. Therefore, {\bf BCK} is sound in the $\mathbb{BCK}$ frame class and this concludes the proof.
\end{proof}

\section{Modal Translation of Substructural Logics}
\label{modal-t-section}
In \cite{pll7} we proved, via translation, that {\bf PLL} (Positive Lattice Logic)  is a fragment of {\bf ML}$_2$.  We defined a translation $\varphi^\bullet$ and a co-translation $\varphi^\circ$ of a sentence $\varphi$ of the  language of {\bf PLL} into the two-sorted modal language {\bf ML}$_2$, by recursion, and we proved that the translation is full and faithful. This section extends the result to any substructural logic.

In Table \ref{trans}, the translation of {\bf PLL} into {\bf ML}$_2$ defined in \cite{pll7} is extended to include the case of substructural logics.

\begin{table}[!htbp]
\caption{Translation and co-translation of {\bf SUB} into {\bf sub.ML}$_2$ }
\label{trans}
\begin{tabbing}
$p_i^\bullet$\hskip1.5cm\==\hskip1mm\= $\lbbox\largediamond P_i$\hskip3.2cm \=  $p_i^\circ$\hskip1.5cm\==\hskip2mm\= $\largesquare\neg P_i$\\
$\top^\bullet$\>=\> $\top$ \> $\top^\circ$ \>=\> ${\tt f}$\\
$\bot^\bullet$ \>=\> $\bot$   \> $\bot^\circ$ \>=\> ${\tt t}$\\
$(\varphi\wedge\psi)^\bullet$ \>=\> $\varphi^\bullet\wedge\psi^\bullet$ \>$(\varphi\wedge\psi)^\circ$\>=\> $\largesquare(\lbdiamond\varphi^\circ\vee\lbdiamond\psi^\circ)$\\
$(\varphi\vee\psi)^\bullet$\>=\> $\lbbox(\largediamond\varphi^\bullet\vee\largediamond\psi^\bullet)$
\> $(\varphi\vee\psi)^\circ$ \>=\> $\varphi^\circ\wedge\psi^\circ$\\[1mm]
$(\varphi\circ \psi)^\bullet$\>=\>$\lbbox\largediamond(\varphi^\bullet\odot \psi^\bullet)$ \> $(\varphi\circ \psi)^\circ$\>=\>$\largesquare\neg(\varphi^\bullet\odot \psi^\bullet)$\\
$(\varphi\ra\psi)^\bullet$ \>=\> $\varphi^\bullet\rfspoon\psi^\bullet$
\>
$(\varphi\ra\psi)^\circ$ \>=\> $\largesquare\neg(\varphi\ra\psi)^\bullet$\\
$(\psi\la\varphi)^\bullet$ \>=\> $\psi^\bullet\lfspoon\varphi^\bullet$
\>
$(\psi\la\varphi)^\circ$ \>=\> $\largesquare\neg(\psi\la\varphi)^\bullet$
\end{tabbing}
\end{table}

\begin{thm}\rm
\label{properties of trans}
Let $\mathfrak{F}$ be a $\mathbb{SUB}$-frame and $\mathfrak{M}=(\mathfrak{F},\imath)$ a {\bf sub.ML}$_2$-model. Define a {\bf SUB}-model $\mathfrak{N}$ on $\mathfrak{F}$ by setting $V(p_i)=\lbbox\largediamond \imath(P_i)$. Then for any {\bf SUB}-sentences $\varphi,\psi$
\begin{enumerate}
\item
$\val{\varphi}_\mathfrak{N}=\val{\varphi^\bullet}_\mathfrak{M}= \val{\lbbox\neg\varphi^\circ}_\mathfrak{M} =\val{\lbbox\largediamond\varphi^\bullet}_\mathfrak{M}$
\item
$\yvval{\varphi}_\mathfrak{N}=\yvval{\varphi^\circ}_\mathfrak{M}= \yvval{\largesquare\neg\varphi^\bullet}_\mathfrak{M} =\yvval{\largesquare\lbdiamond\varphi^\circ}_\mathfrak{M}$
\item
$\varphi\forces\psi$ iff $\varphi^\bullet\models\psi^\bullet$ iff $\psi^\circ\models\varphi^\circ$
\end{enumerate}
\end{thm}
\begin{proof}
Claim 3) is an immediate consequence of the first two, which we prove simultaneously by structural induction. Note that, for 2), the identities $\yvval{\varphi}_\mathfrak{N}= \yvval{\largesquare\neg\varphi^\bullet}_\mathfrak{M} =\yvval{\largesquare\lbdiamond\varphi^\circ}_\mathfrak{M}$ are easily seen to hold for any $\varphi$, given the proof of claim 1), since
\[
\yvval{\varphi}_\mathfrak{N}=\val{\varphi}_\mathfrak{N}^\upv=\largesquare(-\val{\varphi^\bullet}_\mathfrak{M})=
\val{\largesquare\neg\varphi^\bullet}_\mathfrak{M}
\]
\[
\yvval{\varphi}_\mathfrak{N}=\largesquare(-\val{\varphi^\bullet}_\mathfrak{M})= \largesquare(-\val{\lbbox\largediamond\varphi^\bullet}_\mathfrak{M})=
\val{\largesquare\lbdiamond\largesquare\neg\varphi^\bullet}_\mathfrak{M}=
\val{\largesquare\lbdiamond\varphi^\circ}_\mathfrak{M}
\]

For the induction proof, the cases for {\bf PLL} have been covered in \cite{pll7}, Proposition 4.1, to which we refer the reader for details.

\begin{enumerate}
\item[(Case $\bigovert$)] The following calculation proves the claim.
\begin{tabbing}
$\val{\varphi\circ\psi}_\mathfrak{N}$\hskip1mm\==\hskip1mm\= $\val{\varphi}_\mathfrak{N}\bigovert\val{\psi}_\mathfrak{N}$\hskip1.8cm \= \\
\>=\> $\val{\varphi^\bullet}_\mathfrak{M}\bigovert\val{\psi^\bullet}_\mathfrak{M}$ \> By induction\\
\>=\> $\lbbox\largediamond(\val{\varphi^\bullet}_\mathfrak{M}\bigodot(\val{\psi^\bullet}_\mathfrak{M})$\> \\
\>=\> $\lbbox\largediamond\val{\varphi^\bullet\bigodot\psi^\bullet}_\mathfrak{M}$ \> \\
\>=\> $\val{\lbbox\largediamond(\varphi^\bullet\bigodot\psi^\bullet)}_\mathfrak{M}$ \> \\
\>=\> $\val{(\varphi\circ\psi)^\bullet}_\mathfrak{M}$ \> By defn of modal translation
\end{tabbing}
The other two identities for claim 1) are straightforward, given definitions. For 2), observe that $\yvval{\varphi\circ\psi}_\mathfrak{N}=\yvval{\varphi}_\mathfrak{N}\bigovert^{\!\partial}\yvval{\psi}_\mathfrak{N}= \val{\varphi}_\mathfrak{N}^\upv\bigovert^{\!\partial}\val{\psi}_\mathfrak{N}^\upv$ $=(\val{\varphi}_\mathfrak{N}\bigovert\val{\psi}_\mathfrak{N})^\upv=\largesquare(- \val{\lbbox\largediamond(\varphi^\bullet\bigodot\psi^\bullet)}_\mathfrak{M})= \yvval{\largesquare\neg(\varphi^\bullet\bigodot\psi^\bullet)}_\mathfrak{M}$ but the latter is precisely $\yvval{(\varphi\circ\psi)^\circ}_\mathfrak{M}$.

\item[(Case $\Ra$)] By definition of semantics and by the fact that $\Ra$ is the restriction of $\tilde{\Ra}$ on stable sets, see Lemma \ref{sub la ra}, we obtain
\begin{tabbing}
$\val{\varphi\ra\psi}_\mathfrak{N}$\ \==\ \= $\val{\varphi}_\mathfrak{N}\Ra\val{\psi}_\mathfrak{N}$\\
\>=\> $\val{\varphi^\bullet}_\mathfrak{M}\tilde{\Ra}\val{\psi^\bullet}_\mathfrak{M}$ \hskip1.5cm By induction and Lemma \ref{sub la ra}\\
\>=\> $\val{\varphi^\bullet\rfspoon\psi^\bullet}_\mathfrak{M}$\\
\>=\> $\val{(\varphi\ra\psi)^\bullet}_\mathfrak{M}$
\end{tabbing}

\item[(Case $\La$)] Similar to the case for $\Ra$.
\end{enumerate}
Hence for any $\varphi$ we obtained that $\val{\varphi}_\mathfrak{N}=\val{\varphi^\bullet}_\mathfrak{M}$ and the identities $\val{\varphi^\bullet}_\mathfrak{M}= \val{\lbbox\neg\varphi^\circ}_\mathfrak{M} =\val{\lbbox\largediamond\varphi^\bullet}_\mathfrak{M}$ have been verified for the {\bf PLL} fragment in \cite{pll7} and are immediate for $\circ,\la,\ra$ by the way the co-translation was defined.
\end{proof}

To transfer the result to the provability relation,  completeness results need to be established, for each system {\bf SUB} and its companion modal logic {\bf sub.ML}$_2$.

\begin{thm}[Completeness for SUB]\rm
\label{sub completeness}
Every substructural system {\bf SUB} of Figure \ref{sub-hierarchy} is (sound, by Theorem \ref{sub soundness}, and) complete in the corresponding frame class $\mathbb{SUB}$, as defined in Definition \ref{frame classes}.
\end{thm}
\begin{proof}
For completeness, let $\mathcal{L}$ be the Lindenbaum-Tarski algebra of the logic and let $\mathfrak{F}=(\filt(\mathcal{L}),\upv,\idl(\mathcal{L}))$ be the canonical lattice frame \cite{sdl}, where $\filt(\mathcal{L})$ is the complete lattice of (proper) lattice-filters, $\idl(\mathcal{L})$ is the complete lattice of (proper) lattice-ideals and $x\upv y$ iff $x\cap y\neq\emptyset$. Note that the canonical frame is {\em separated}   i.e. the pre-order induced by setting $x\preceq z$ iff $\{x\}\rperp\subseteq\{z\}\rperp$ is a partial order on $X$ and similarly for $Y$. This is because if $\{x\}\rperp\subseteq\{z\}\rperp$, then for any $a\in x$, $x\upv y_a$, hence $z\upv y_a$ which implies that $a\in z$. Furthermore, it follows from the definition of $\upv$ that it is increasing in each argument place.

Following \cite{dloa} and as detailed  in \cite{sdl-exp,discres}, define the canonical relation $R^{111}$  by first defining a point operator $x\widehat{\circ}z=\bigvee\{x_{a\circ b}\midsp a\in x, b\in z\}$ on filters, where $x,z,z'\in X$ and $x_e=e\!\uparrow$ designates a principal filter  and then letting $xR^{111}zz'\mbox{ iff }z\widehat{\circ}z'\leq x$, where $\leq$ designates filter inclusion.  Note  that for any $x,z,z'$ we easily obtain from the definition that
\[
xR^{111}zz'\;\mbox{ iff }\;\forall a,b\;(a\in z\;\wedge\;b\in z'\;\lra\;a\circ b\in x)
\]
It remains to show (a) that the canonical frame for a substructural logic {\bf SUB} is a $\mathbb{SUB}$-frame and that (b) the canonical (co)interpretation, defined by  $\val{\varphi}=\{x\in\filt(\mathcal{L})\midsp [\varphi]\in x\}$ and $\yvval{\varphi}=\{y\in\idl(\mathcal{L})\midsp [\varphi]\in y\}$ respects the recursive satisfaction conditions of Table \ref{sat}. We do this next in two lemmas, the canonical frame and canonical interpretation lemmas.
\end{proof}

\begin{lemma}[Canonical SUB-Frame Lemma]\rm
\label{sub frame lemma}
The canonical frame for a substructural logic {\bf SUB} is a $\mathbb{SUB}$-frame.
\end{lemma}
\begin{proof}
There is nothing to prove for {\bf NFL}. We argue separately each case of a logic above {\bf NFL} (see Figure \ref{sub-hierarchy}).\\

\paragraph{FL} To show that for all points $w, z, u, v\in X$, statements (1) and (2) below are equivalent in the canonical frame
    \begin{quote}
      (1) $\exists x\;(xR^{111}uv\;\wedge\;zR^{111}xw)$\hskip1cm
      (2) $\exists x\;(xR^{111}vw\;\wedge\;zR^{111}ux)$
    \end{quote}
note first that equivalence of (1), (2) is an immediate consequence of associativity of the point operator $\widehat{\circ}$. Indeed, assuming (1) holds, i.e. for some $x\in X$ both $u\widehat{\circ}v\leq x$ and $x\widehat{\circ}w\leq z$ hold, then $(u\widehat{\circ} v)\widehat{\circ}w\leq x\widehat{\circ}w\leq z$ and we may take $x'=v\widehat{\circ}w$. Similarly for the converse.

To prove associativity of the point operator $\widehat{\circ}$ note that
it follows from \cite{dloa} (Theorem 6.6) where point operators were first introduced that $\widehat{\circ}$ distributes over arbitrary joins in each argument place, given that the lattice operator $\circ$ distributes over binary joins and that (\cite{dloa}, Lemma 6.7) it preserves principal filters, i.e. $x_a\widehat{\circ}x_b=x_{a\circ b}$. It is then obtained  that

\begin{tabbing}
$u\widehat{\circ}(v\widehat{\circ}w)$ \ \==\ \= $(\bigvee_{a\in u}x_a)\widehat{\circ}((\bigvee_{b\in v}x_b)\widehat{\circ}(\bigvee_{c\in w}x_c))$\ \==\ \= $\bigvee_{a\in u,b\in v,c\in w}(x_a\widehat{\circ}(x_b\widehat{\circ}x_c))$\\
\>=\> $\bigvee_{a\in u,b\in v,c\in w}x_{a\circ(b\circ c)}$ \>=\> $\bigvee_{a\in u,b\in v,c\in w}x_{(a\circ b)\circ c}$\\
\>=\> $\cdots$\>=\> $(u\widehat{\circ}v)\widehat{\circ}w$
\end{tabbing}
and this establishes associativity of $\widehat{\circ}$. This implies, incidentally, that the associativity equation is canonical.
Indeed, recall  that the operator $\bigovert$, defined on stable sets by $A\bigovert C=\lbbox\largediamond(A\bigodot C)$, was shown in \cite{sdl-exp,discres} to be the $\sigma$-extension $\circ^\sigma$ of the lattice cotenability operator $\circ$. From \cite{mai-harding}, recall that $\circ^\sigma$ is first defined on closed elements of $\gpsi$ (sets of the form $\Gamma x=\{z\in X\midsp x\leq z\}$) by setting
\[
\Gamma z\circ^\sigma\Gamma z'=\bigwedge\{\Gamma x_{a\circ b}\midsp a\in z,b\in z'\}
\]
then extended to all stable sets using join-density of closed elements by setting $A\circ^\sigma C=\bigvee_{z\in A, z'\in C}(\Gamma z\circ^\sigma\Gamma z')$. As pointed out in \cite{sdl-exp,canonical,discres,discr}, this delivers the same map on stable sets as the map defined on closed elements by setting $\Gamma z\medcircle\Gamma z'=\Gamma(z\widehat{\circ}z')$, then again extending to all stable sets by using join-density of closed elements. This is immediate since $\bigwedge\{\Gamma x_{a\circ b}\midsp a\in z,b\in z'\}=\Gamma\left(\bigvee\{x_{a\circ b}\midsp a\in z,b\in z'\}\right)$, where by the definition of point operators in \cite{dloa} we have $\bigvee\{x_{a\circ b}\midsp a\in z,b\in z'\}=z\widehat{\circ}z'$. Thereby, the identity
 $A\bigovert(C\bigovert F)=(A\bigovert C)\bigovert F$ obtains on stable sets iff the identity $\Gamma u\bigovert(\Gamma v\bigovert\Gamma w)=(\Gamma u\bigovert\Gamma v)\bigovert\Gamma w$ on closed elements holds. Given the fact pointed out above that $\Gamma z\bigovert\Gamma z'=\Gamma z\medcircle\Gamma z'=\Gamma(z\widehat{\circ}z')$, associativity of the filter operator $\widehat{\circ}$ implies canonicity of the associativity equation since
$
\Gamma u\mbox{$\bigovert$}(\Gamma v\mbox{$\bigovert$}\Gamma w)=(\Gamma u\mbox{$\bigovert$}\Gamma v)\mbox{$\bigovert$}\Gamma w\mbox{ iff }\Gamma(u\widehat{\circ}(v\widehat{\circ}w))=\Gamma((u\widehat{\circ}v)\widehat{\circ}w)
$, which is equivalent to the identity $u\widehat{\circ}(v\widehat{\circ}w)=(u\widehat{\circ}v)\widehat{\circ}w$, given that the canonical lattice frame is separated.\\

\paragraph{BCI} The case is straightforward, since by definition of $\;\widehat{\circ}$, if $\circ$ is commutative, then so is the  operator $\widehat{\circ}$: $z\widehat{\circ}z'=\bigvee\{x_{a\circ b}\midsp a\in z, b\in z'\}=\bigvee\{x_{b\circ a}\midsp a\in z, b\in z'\}=z'\widehat{\circ}z$. Hence $xR^{111}zz'$ holds iff $xR^{111}z'z$ holds, i.e. the constraint (C2) holds in the canonical frame.\\

\paragraph{BCK}  Given $a\circ b\leq b$, we obtain $x_b\leq x_{a\circ b}$ for any $a$. Hence
\[
z'=\bigvee_{b\in z'}x_b\leq \bigvee_{b\in z'}x_{a\circ b}\leq\bigvee_{a\in z,b\in z'}x_{a\circ b}=z\widehat{\circ}z'
\]
Therefore, if  $xR^{111}zz'$, i.e. $z\widehat{\circ}z'\leq x$, it follows that $z'\leq z\widehat{\circ}z'\leq  x$ and thereby the constraint (C3) holds in the canonical frame.\\

\paragraph{BCW} We now have $a\wedge b\leq a\circ b$, i.e. $x_{a\circ b}\leq x_{a\wedge b}$. Then
\[
z\widehat{\circ}z'=\bigvee_{a\in z,b\in z'}x_{a\circ b}\leq\bigvee_{a\in z,b\in z'}x_{a\wedge b}=\bigvee_{a\in z,b\in z'}(x_a\vee x_b)=\left(\bigvee_{a\in z}x_a\right)\vee\left(\bigvee_{b\in z'}x_b\right)=z\vee z'
\]
In particular, $x\widehat{\circ}x\leq x\vee x=x$, i.e. $xR^{111}xx$ holds. Hence the constraint (C4) holds in the canonical frame of {\bf BCW}.
\end{proof}

\begin{lemma}[Canonical SUB-Interpretation Lemma]\rm
\label{sub int lemma}
The canonical interpretation,  assigning to a sentence $\varphi$ the formal concept $(\val{\varphi},\yvval{\varphi})$ with extent $\val{\varphi}=\{x\in X\midsp [\varphi]\in x\}$ and intent $\yvval{\varphi}=\{y\in Y\midsp [\varphi]\in y\}$, where $[\varphi]$ is the equivalence class of $\varphi$ under provability,  satisfies the recursive conditions of Table \ref{sat}.
\end{lemma}
\begin{proof}
The proof is by induction, the cases for $\top,\bot,\wedge,\vee$ are straightforward and we turn to the cases for the residuated operators.

The canonical accessibility relation $R^{111}$ has been defined by setting, for filters $x,u,v$, $xR^{111}uv$ iff $u\widehat{\circ}v\leq x$ (where we consistently use $\leq$ for filter inclusion). Hence $R^{111}uv=\{x\midsp u\widehat{\circ}v\leq x\}=\Gamma(u\widehat{\circ}v)$ is a Galois stable set, in fact a closed element of $\gpsi$. The canonical Galois relation $\upv$ is increasing in each argument place from which it follows that $yR^{\partial 11}uv$ iff $u\widehat{\circ}v\upv y$. It needs to be verified, given the co-satisfaction clause for $\circ$, that $y\in\yvval{\varphi\circ\psi}$ iff for any $u,v\in X$, if $u\in\val{\varphi}$ and $v\in\val{\psi}$, then $yR^{\partial 11}uv$. By definition of the canonical interpretation we have $y\in\yvval{\varphi\circ\psi}$  iff\ $[\varphi\circ\psi]\in y$  iff  $[\varphi]\circ[\psi]\in y$. For lattice elements $a,b$, given monotonicity of $\widehat{\circ}$ and increasingness of $\upv$ we obtain
\begin{tabbing}
$a\circ b\in y$\ \= iff\ \= $x_{a\circ b}\upv y$\ \= iff\ \= $x_a\widehat{\circ}x_b\upv y$\ \= iff \ \=
$\forall u,v\in X\;(x_a\leq u\;\wedge\;x_b\leq v\;\lra\; u\widehat{\circ}v\upv y)$\\
\>iff\> $\forall u,v\in X\;(a\in u\;\wedge\;b\in v\;\lra\;yR^{\partial 11}uv)$
\end{tabbing}
In particular, for $a=[\varphi], b=[\psi]$ the above yields precisely the co-satisfaction clause for cotenability:
$y\dforces\varphi\circ\psi\mbox{ iff }
 \forall u,v\in X\;(u\forces\varphi\;\wedge\; v\forces\psi\;\lra\; yR^{\partial 11} uv)$.

For the satisfaction clause for implication, it suffices to show that \mbox{$a\ra b\in z$} iff $\forall x\in X\;\forall y\in Y\;(a\in x\;\wedge\;b\in y\;\lra\; yR^{\partial 11}zx)$.

Define first a point operator $x\onlra{\flat}y$ as in \cite{dloa}, given a filter $x$ and an ideal $y$, to be the ideal generated by all implications $a\ra b$ with $a\in x, b\in y$, i.e. $x\onlra{\flat}y=\bigvee\{y_{a\ra b}\midsp a\in x, b\in y\}$.

Note that, it follows from the results of \cite{sdl-exp} that for the particular case of the normal lattice operator $\circ:\mathcal{L}\times\mathcal{L}\lra\mathcal{L}$, of distribution type $(1,1;1)$, the operator $\bigovert$ on $\gpsi$ defined by $A\bigovert C=\bigvee_{x\in A,z\in C}(\Gamma x\bigovert\Gamma z)=\bigvee_{x\in A,z\in C}\Gamma(x\widehat{\circ}z)$ is its  $\sigma$-extension and that $A\bigovert C=\lbbox\largediamond(A\bigodot C)$.

In \cite{sdl-exp} and more specifically in \cite{canonical} we have also constructed the  $\sigma$-extension $\Ra_1=\ra^\sigma$ of the lattice implication operator $\ra$ in our canonical lattice frame. This was done by first defining   $\Ra_0:\gpsi\times\gphi\lra\gphi$  by $\Gamma x\Ra_0\Gamma y=\Gamma(x\onlra{\flat}y)$. This is a monotone operator, completely distributing over arbitrary joins of closed elements in each argument place, hence it is extended to all of $\gpsi\times\gphi$ by setting $A\Ra_0 B=(\bigvee_{x\in A}\Gamma x)\Ra_0(\bigvee_{y\in B}\Gamma y)$ $=\bigvee_{x\in A,y\in B}(\Gamma x\Ra_0\Gamma y)$. Then
$\Ra_1:\gpsi\times\gpsi^\partial\lra\gpsi^\partial$ was defined in \cite{canonical} by setting $A\Ra_1 C=\lperp(A\Ra_0 C\rperp)$. In particular, $\Gamma x\Ra_1\lperp\{y\}=\lperp(\Gamma x\Ra_0\Gamma y)= \lperp\{x\onlra{\flat}y\}$.

By results of \cite{mai-harding}, Proposition 6.6, it follows that residuation is preserved by canonical extensions, hence for stable sets $A,C,F$ it holds that $A\bigovert F\subseteq C$ iff $C\subseteq A\Ra_1 C$. In particular then, $\Gamma x\bigovert\Gamma z\subseteq\lperp\{y\}$ iff $\Gamma z\subseteq\Gamma x\Ra_1\lperp\{y\}$. This is equivalent, given the definitions of the stable set operators, to $x\widehat{\circ}z\upv y$ iff $z\upv x\onlra{\flat}y$.
Given the definition of $R^{\partial 11}$, we obtain
\begin{equation}
yR^{\partial 11}xz\;\mbox{ iff }\; x\widehat{\circ}z\upv y \;\mbox{ iff }\;z\upv x\onlra{\flat}y\label{canonical relation and upv}
\end{equation}
Furthermore, using the above we obtain
\begin{tabbing}
$a\ra b\in x$\ \=\ iff\ \=$x\upv y_{a\ra b}$ iff $x\upv (x_a\onlra{\flat}y_b)$\\
\>iff\> $\forall z\in X\;\forall y\in Y\;(x_a\leq z\;\wedge\;y_b\leq y\;\lra\;x\upv (z\onlra{\flat}y))$\\
\>iff\> $\forall z\in X\;\forall y\in Y\;(a\in z\;\wedge\;b\in y\;\lra\;x\upv (z\onlra{\flat}y))$\\
\>iff\> $\forall z\in X\;\forall y\in Y\;(a\in z\;\wedge\;b\in y\;\lra\;z\widehat{\circ}x\upv y)$ \ \ using \eqref{canonical relation and upv}\\
\>iff\> $\forall z\in X\;\forall y\in Y\;(a\in z\;\wedge\;b\in y\;\lra\;yR^{\partial 11}zx)$
\end{tabbing}

By a completely symmetric argument and with $y\stackrel{\flat}{\lla} z$ defined as the join $\bigvee\{y_{a\circ b}\midsp a\in z, b\in y\}$ we similarly obtain that  $x\widehat{\circ}z\upv y$ iff $x\upv y\stackrel{\flat}{\lla} z$, from which it follows by the same argument as above that $b\la a\in x$ iff $\forall z\in x$ $\forall y\in Y\;(a\in z\;\wedge\; b\in y\;\lra\; yR^{\partial 11}xz)$ and this completes the proof of the canonical interpretation lemma.
\end{proof}

We may then conclude.
\begin{thm}[SUB Soundness and Completeness]\rm
Each substructural logic {\bf SUB} of Figure \ref{sub-hierarchy} is sound and complete in the respective $\mathbb{SUB}$ frame-class.\telos
\end{thm}

For completeness of the modal companion logics, recall that to show that if $\alpha$ (resp. $\beta$) is logically valid, i.e. $\models\alpha$ (resp. $\models\beta$), then $\alpha$ is a theorem, i.e. $\proves\alpha$ (resp. $\vproves\beta$), one proceeds contrapositively arguing that for every non-theorem $\alpha$ (resp. $\beta$) a model can be found such that $\alpha$ (resp. $\beta$) is not valid in that model. The well established completeness argument proceeds by demonstrating that there is in fact a single model that invalidates all non-theorems of the logic, the so-called canonical model $\mathfrak{M}_c=(\mathfrak{F}_c,\val{\;}_c)$ of the logic, where $\mathfrak{F}_c=(W,\upv_w,W')$ is the canonical (sorted, in our case) frame and $\val{\zeta}_c=\{w\in W\midsp [\zeta]\in w\}$, while $\yvval{\eta}=\largesquare(W\setminus\val{\lbbox\neg\eta}_c)$ and where $[\zeta]$ is the equivalence class, under provability, of $\zeta$.

Algebraically, the canonical frame construction is a realization of a canonical extension of the Lindenbaum-Tarski algebra of the logic.
The standard (ultrafilter) canonical frame of a modal companion logic {\bf sub.ML}$_2$ is constructed from ultrafilters, as detailed  for {\bf ML}$_2$. The canonical accessibility relation $R^{111}$ is defined on utrafilters by the condition $xR^{111}zz'$ iff for all $a\in z$ and $a'\in z'$ we have $a\circ a'\in x$. We will however have use, for proof purposes, of an alternative canonical extension which regards the Boolean algebra merely as a lattice. By the results of Section \ref{canext section} the two canonical extension constructions are isomorphic, a fact that we use in the proofs of the canonical frame and interpretation lemmas.

\begin{lemma}[Canonical sub.ML$_2$-Frame Lemma]\rm
\label{subML frame lemma}
The canonical frame of {\bf sub.ML}$_2$ is a $\mathbb{SUB}$-frame.
\end{lemma}
\begin{proof}
There is nothing to prove for the system {\bf nfl.ML}$_2$ (the canonical relation $R^{111}$ generates residuated operators as detailed in Section \ref{operators on sets section}) and we discuss each other system separately.\\
\paragraph{fl.ML$_2$}
This case (associativity relational constraint) has been identified as a non-elementary case and resolved by logicians working on Relevance Logic \cite{relevance1,relevance2}  and we could simply refer to the related results (see, for example, a detailed proof in \cite{relevance3}). However, a simple solution can be now given using the notion of a canonical extension and related results and we report it below. Note first that, algebraically, the canonical frame is a realization of the construction of a canonical extension of the Lindenbaum-Tarski algebra of the logic (a residuated Boolean algebra, considering only the first sort, where the associative operation $\circ$ is defined). For distributive lattices the extension is based on the set of prime filters (as is the case for relevance logic) and for the Boolean case ultrafilters are considered, as it is well known. For arbitrary lattices all filters are used \cite{sdl}. We already mentioned (and proved) in Section \ref{sub completeness} that the associativity equation $a\circ (b\circ c)=(a\circ b)\circ c$ is canonical (it is preserved by canonical extensions). Furthermore, the set operator (on sets of ultrafilters) defined by
\[
U_1\mbox{$\bigodot$} U_2=\{u\in\ufilt(\mathcal{A})\midsp \exists u_1,u_2\in\ufilt(\mathcal{A})\;(u_1\in U_1\;\wedge\;u_2\in U_2\;\wedge\; uRu_1u_2)\}
\]
is the $\sigma$-extension of $\circ$ (by the results of J\'{o}nsson and Tarski on Boolean algebras with operators \cite{jt1,jt2}), hence associative, where the canonical accessibility relation $R$ is defined (as for relevance logic) by
\[
uRu_1u_2\;\mbox{ iff }\;\forall a,b\;(a\in u_1\;\wedge\;b\in u_2\;\lra\;a\circ b\in u)
\]
It is furthermore elementary to see that the associativity relational constraint (C1) is equivalent to the identity $\{u\}\bigodot(\{v\}\bigodot\{w\})=(\{u\}\bigodot\{v\})\bigodot\{w\}$, for any $u,v,w\in\ufilt(\mathcal{A})$. Hence (C1) holds in the canonical {\bf fl.ML}$_2$ frame.\\

\paragraph{bci.ML$_2$}
Straightforward, since assuming $xR^{111}zz'$, if $a\in z,b\in z'$, then $a\circ b\in x$, hence since commutativity is assumed in this case we also get $b\circ a\in x$ from which $xR^{111}z'z$ follows, by definition of the relation.
\\

\paragraph{bck.ML$_2$}
Let $x,z,z'$ be ultrafilters such that $xR^{111}zz'$ holds. Fix $a\in z$, arbitrary. For any $b\in z'$, $a\circ b\in x$ and since $a\circ b\leq b$ holds in the Lindenbaum-Traksi algebra of the logic and $x$ is a filter, we obtain that $b\in x$. It follows that $z'\subseteq x$, which implies that $z'=x$, since they are both maximal filters. Since $\preceq$ is a preorder, hence reflexive, it follows from $x\preceq x$ that $z'\preceq x$. Hence the relational constraint (C3) holds in the canonical frame.
\\

\paragraph{bcw.ML$_2$}
The case is straightforward since, given any ultrafilter $x$, for any lattice elements $a,b$, if both $a,b\in x$, then $a\wedge b\in x$ and since contraction is assumed in {\bf bcw.ML}$_2$ we get $a\circ b\in x$ from $a\wedge b\leq a\circ b$. Hence (C4) holds in the canonical frame.
\end{proof}

\begin{lemma}[Canonical sub.ML$_2$-Interpretation Lemma]\rm
\label{subML int lemma}
The canonical interpretation of {\bf sub.ML}$_2$ satisfies the recursive clauses of Table \ref{sat modal}.
\end{lemma}
\begin{proof}
For each sort of sentences, the canonical interpretation $\val{\;}_c,\yvval{\;}_c$ assigns to a sentence the set of respective ultrafilters containing its equivalence class, as usual. The proof is by induction and we skip the familiar cases corresponding to the Boolean operators $\neg,\wedge$ (for each sort). For the needs of the proof we let $a=[\alpha]$, primed or with subscripts, and $b=[\beta]$,  again possibly with subscripts or primed.\\

\paragraph{Case $\largesquare\alpha,\lbbox\beta$}\mbox{}\\
Both cases are immediate by the way the canonical frame relation $I$ is defined.\\

\paragraph{Case $\odot$}\mbox{}\\
The current proof obligation is known in the relevance logic literature and it has been handled by Routley and Meyer \cite{routley-meyer} (see also \cite{relevance2,relevance3}) with a maximalization argument using Zorn's lemma (delivering prime filters, which are maximal filters in our Boolean algebra case) and we could simply direct the reader to that proof argument. Instead, we provide a proof that makes use of the uniqueness, up to isomorphism, of canonical extensions of lattice expansions, see Lemma \ref{eta morphism}.

Recall from the results of J\'{o}nsson and Tarski \cite{jt1,jt2} that the additive, binary operator $\circ$ on the Boolean algebra $\mathcal{A}$ is represented (in the ultrafilter canonical extension) relationally as the operator
\[
U\medcircle U'=\{u\in\ufilt(\mathcal{A})\midsp \exists w,w'\in\ufilt(\mathcal{A})\;(uRww'\wedge w\in U\wedge w'\in U')\}
\]
where $R$ is defined by
\[
uRww'\;\mbox{ iff }\;\forall a,a'\in\mathcal{A}\;(a\in w\;\wedge\;a'\in w'\;\lra\;a\circ a'\in u)
\]
Furthermore,  $\medcircle=\circ^\sigma_u$ is the $\sigma$-extension (in the terminology of \cite{mai-harding}) of the cotenability operator $\circ$ and it is a completely additive operator, i.e. $U\circ^\sigma_u U'=\bigcup_{w\in U, w'\in U'}\{w\}\circ^\sigma_u\{w'\}$. Note also that $uRww'$ iff $u\in\{w\}\circ^\sigma_u\{w'\}$.

Now regard the Boolean algebra merely as a lattice and take its canonical extension as the complete lattice of Galois stable sets $\gpsi$, where $X=\filt(\mathcal{A})$. Let $X\supseteq X^u=\ufilt(\mathcal{A})$. As shown in Section \ref{canext section}, the isomorphism (by uniqueness of canonical extensions of lattices) between $\gpsi$ and $\powerset(X^u)$ sends a stable set $A\in\gpsi$ to $A^u=A\cap X^u$. The cotenability operator $\circ$ is now represented as an operator on Galois stable sets as in the proof of Lemma \ref{sub frame lemma}, by first defining a point operator $\widehat{\circ}$ (on filters), setting $z\widehat{\circ}z'=\bigvee\{x_{a\circ a'}\midsp a\in z, a'\in z'\}$, then defining $R^{111}\subseteq X^3$ by $xR^{111}zz'$ iff $z\widehat{\circ}z'\leq x$, which is equivalent to the definition $xR^{111}zz'$ iff $\forall a,a'\in\mathcal{A}\;(a\in z\;\wedge\;a'\in z'\;\lra\;a\circ a'\in x)$. Note that the restriction $R=R^{111}\cap (X^u)^3$ on the set of ultrafilters is precisely the canonical accessibility relation $R$ in the standard canonical (ultrafilter) extension of the Boolean algebra with an additive operator $\circ$. By the analysis in Section \ref{canext section}, $\eta(\Gamma x\circ^\sigma\Gamma z)=\eta(\Gamma(x\widehat{\circ}z))=\{u\in X^u\midsp x\widehat{\circ}z\leq u\}$ which is precisely the set $\{u\in X^u\midsp uR^{111}xz\}$. By Lemma \ref{eta morphism}, $\eta(\Gamma x\circ^\sigma\Gamma z)=\eta(\Gamma x)\circ^\sigma_u\eta(\Gamma z)=(\Gamma^ux)\circ^\sigma_u(\Gamma^uz)$, using the notation introduced in the above mentioned lemma. By complete additivity of $\circ^\sigma_u$ the latter is the union $\bigcup_{x\leq w, z\leq w'}(\{w\}\circ^\sigma_u\{w'\})$. Hence we have obtained that, for an ultrafilter $u$ and any filters $x,z$,
\begin{tabbing}
$uR^{111}xz$\hskip6mm \= iff \ \= $\exists w,w'\in X^u\;(x\leq w\;\wedge\;z\leq w'\;\wedge\;u\in(\{w\}\circ^\sigma_u\{w'\})$\\
\>iff\> $\exists w,w'\in X^u\;(x\leq w\;\wedge\;z\leq w'\;\wedge\;uRww')$
\end{tabbing}

In particular, given elements $a,a'$ of the Boolean algebra,
\begin{equation}
uR^{111}x_ax_{a'} \;\mbox{ iff }\; \exists w,w'\in X^u\;(a\in w\;\wedge\;a'\in w'\;\wedge\;uRww')
\label{r111-r}
\end{equation}

Now let $a=[\alpha], a'=[\alpha']$ be the respective equivalence classes of $\alpha,\alpha'$. The canonical interpretation $\val{\alpha\odot\alpha'}_c$ is the representation of the equivalence class $[\alpha\odot\alpha']=[\alpha]\circ[\alpha']$, i.e. $\val{\alpha\odot\alpha'}_c=h(a\circ a')$, where $h$ is the Stone representation map $h(e)=\{u\in X^u=\ufilt(\mathcal{A})\midsp e\in u\}$.
Assume for an ultrafilter $u$ that $u\in\val{\alpha\odot\alpha'}_c$, so that $a\circ a'\in u$. It needs to be shown that there exist ultrfilters $w,w'$ such that $a\in w, a'\in w'$ and $uRww'$. But $a\in x_a, a'\in x_{a'}$ and if $a\circ a'\in u$, then $x_a\widehat{\circ}x_{a'}=x_{a\circ a'}\leq u$, i.e. $uR^{111}x_ax_{a'}$ holds. This is equivalent, by \eqref{r111-r}, to the assertion that ultrafilters $w,w'$ exist with $a\in w, a'\in w'$ and where $uRww'$. This shows that for an ultrafilter $u$,  \[a\circ a'\in u\;\mbox{ iff }\;\exists w,w'\in X^u\;(a\in w\;\wedge\;a'\in w'\;\wedge\;uRww')\] hence the canonical interpretation conforms to the relevant clause in the definition of the satisfaction relation of Table \ref{sat modal}.
\\

\paragraph{Case $\rfspoon$, $\lfspoon$}\mbox{}\\
The cases of $\lfspoon,\rfspoon$ follow from the case for $\odot$, given Lemma \ref{modal residuals lemma} and hence the proof is complete.
\end{proof}

Hence we may conclude
\begin{thm}[Completeness for sub.ML$_2$]\rm Each modal companion logic {\bf sub.ML}$_2$ is (sound, by Theorem \ref{subML2 soundness}) and complete in the respective $\mathbb{SUB}$ frame-class.\telos
\end{thm}

By Theorem \ref{properties of trans} and by the soundness and completeness results proven above for each substructural logic {\bf SUB} and its companion modal logic {\bf sub.ML}$_2$,  we obtain  that
\begin{coro}\rm
$\varphi\proves\psi$ in {\bf SUB} iff $\varphi^\bullet\proves\psi^\bullet$ iff $\psi^\circ\vproves\varphi^\circ$ in {\bf sub.ML}$_2$.\telos
\end{coro}

\section{Translation into Sorted First-Order Logic}
\label{fol trans}
Assume $V_x,V_y$ are countable, non-empty and disjoint sets and let
\[
{\mathcal  L}^1(V_x,V_y,{\bf P}^x_i, {\bf Q}^y_i, {\bf I}^{xy}, {\bf R}^{xxx}, =_x,=_y)
\]
be the two-sorted, first-order language with equality whose sorted variables are in $V_x$ (ranged over by $x,z$ etc),$V_y$ (ranged over by $y,v$ etc) and it includes sorted quantifiers $\forall^x,\forall^y$, a binary predicate {\bf I}$^{xy}$, unary predicates {\bf P}$^x_i$, {\bf Q}$^y_i$ (where sorting is indicated by the superscript) corresponding to the propositional variables of {\bf sub.ML}$_2$ and a three-place predicate ${\bf R}^{xxx}$.   Superscripts indicating sorting type will be dropped in the sequel, for simplicity of notation. A two-sorted first-order structure for ${\mathcal  L}^1$ is a structure $\langle X, Y, {\bf I}^I, {\bf P}^I_i, {\bf Q}^I_i, {\bf R}^I)$ where $X,Y$ are disjoint sets and ${\bf P}^I_i\subseteq X$, ${\bf R}^I\subseteq X^3$, ${\bf Q}^I_i\subseteq Y$, for each $i$, and ${\bf I}^I\subseteq X\times Y$, while $=^I_x,=^I_y$ are the identity relations for each sort.

The standard translation of {\bf sub.ML}$_2$ into {\bf FOL} (First-Order Logic) is exactly as in the single-sorted case, except for the relativization to two sorts, displayed in Table \ref{std-ML2}.

\begin{table}[!htbp]
\caption{Standard Translation of {\bf sub.ML}$_2$ into 2-Sorted {\bf FOL}}
\label{std-ML2}
\begin{tabbing}
$\stx{x}{P_i}$\hskip0.8cm\==\hskip2mm\= ${\bf P}_i(x)$ \\
$\stx{x}{\neg \alpha}$\>=\> $\neg\stx{x}{\alpha}$ \\
$\stx{x}{\alpha\wedge \alpha'}$\>=\> $\stx{x}{\alpha}\wedge\stx{x}{\alpha'}$ \\
$\stx{x}{\lbbox \beta}$ \>=\> $\forall y\;({\bf I}(x,y)\;\lra\;\sty{y}{\beta})$ \\
$\stx{x}{\alpha\odot\alpha'}$ \>=\> $\exists z,z'\;({\bf R}(x,z,z')\;\wedge\;\stx{z}{\alpha}\;\wedge\;\stx{z'}{\alpha'})$\\
$\stx{x}{\alpha\rfspoon\alpha'}$ \>=\> $\forall z,z'\;(\stx{z}{\alpha}\;\wedge\;{\bf R}(z',z,x)\;\lra\;\stx{z'}{\alpha'})$\\
$\stx{x}{\alpha'\lfspoon\alpha}$ \>=\> $\forall z,z'\;(\stx{z}{\alpha}\;\wedge\;{\bf R}(z',x,z)\;\lra\;\stx{z'}{\alpha'})$
\\[3mm]
$\sty{y}{Q_i}$\>=\> ${\bf Q}_i(y)$\\
$\sty{y}{\neg \beta}$ \>=\> $\neg\sty{y}{\beta}$\\
$\sty{y}{\beta\wedge \beta'}$\>=\> $\sty{y}{\beta}\wedge\sty{y}{\beta'}$\\
$\sty{y}{\largesquare \alpha}$ \>=\> $\forall x\;({\bf I}(x,y)\;\lra\;\stx{x}{\alpha})$
\end{tabbing}
\end{table}

Note that a {\bf sub.ML}$_2$-model $\mathfrak{M}=((X,I,Y, R^{111}),V)$ is also an ${\mathcal  L}^1$-structure, setting ${\bf I}^I=I, {\bf R}^I=R^{111}, {\bf P}^I_i=V(P_i), {\bf Q}^I_i=V(Q_i)$, so that the following is immediate.

\begin{prop}\rm
\label{ML2-L1}
For any {\bf sub.ML}$_2$ formulae $\alpha,\beta$ (of sort 1, 2, respectively),  {\bf sub.ML}$_2$-model $\mathfrak{M}=((X,I,Y,R^{111}),V)$ and  $x\in X, y\in Y$, $x\models_\mathfrak{M} \alpha$ iff $\mathfrak{M}\models\stx{z}{\alpha}[z:=x]$ and $y\models_\mathfrak{M} \beta$ iff $\mathfrak{M}\models\sty{v}{\beta}[v:=y]$.\telos
\end{prop}
Composing with the translation and co-translation of {\bf SUB} into {\bf sub.ML}$_2$, respective translations of {\bf SUB} into ${\mathcal  L}^1$ are obtained, displayed in Table \ref{std-PLL}.
\begin{table}[!htbp]
\caption{Standard (2-Sorted) First-Order Translation and Co-Translation of {\bf SUB}}
\label{std-PLL}
\begin{tabbing}
$\stxb{x}{p_i}$ \hskip1.5cm\==\hskip2mm\= $\forall y\;({\bf I}(x,y)\lra\;\exists z\;({\bf I}(z,y)\;\wedge\;{\bf P}_i(z)))$\\
\hskip3mm $\styc{y}{p_i}$ \>=\>\hskip3mm  $\forall x\;({\bf I}(x,y)\;\lra\;\neg{\bf P}_i(x))$\\
$\stxb{x}{\top}$ \>=\> $x=x$\\
\hskip3mm $\styc{y}{\top}$ \>=\>\hskip3mm  $\forall x\;({\bf I}(x,y)\;\lra\; x\neq x)$\\
$\stxb{x}{\bot}$ \>=\> $\forall y\;({\bf I}(x,y)\;\lra\; y\neq y)$\\
\hskip3mm $\styc{y}{\bot}$ \>=\>\hskip3mm  $y=y$\\
$\stxb{x}{\varphi\wedge\psi}$ \>=\> $\stxb{x}{\varphi}\wedge\stxb{x}{\psi}$\\
\hskip3mm $\styc{y}{\varphi\wedge\psi}$ \>=\>\hskip3mm  $\forall x\;({\bf I}(x,y)\;\lra\;\exists v\;({\bf I}(x,v)\;\wedge\;(\styc{v}{\varphi}\vee\styc{v}{\psi})))$\\
$\stxb{x}{\varphi\vee\psi}$ \>=\>   $\forall y\;({\bf I}(x,y)\;\lra\;\exists z\;({\bf I}(z,y)\;\wedge\; (\stxb{z}{\varphi}\vee\stxb{z}{\psi})))$ \\
\hskip3mm $\styc{y}{\varphi\vee\psi}$ \>=\>\hskip3mm  $\styc{y}{\varphi}\wedge\styc{y}{\psi}$
\\
$\stxb{x}{\varphi\ra\psi}$ \>=\> $\forall z,z'\;(\stxb{z}{\alpha}\;\wedge\;{\bf R}(z',z,x)\;\lra\;\stxb{z'}{\alpha'})$\\
\hskip3mm $\styc{y}{\varphi\ra\psi}$ \>=\> $\forall x\;({\bf I}(x,y)\lra\neg\stxb{x}{\varphi\ra\psi})$
\\
$\stxb{x}{\varphi\circ\psi}$ \>=\> $\forall y\;({\bf I}(x,y)\;\lra\;\exists x'\;({\bf I}(x',y)\;\wedge\;\exists z,z'\;({\bf R}(x',z,z')\;\wedge$\\
\>\>\hskip3.5cm $\wedge\;\stxb{z}{\alpha}\;\wedge\;\stxb{z'}{\alpha'})$\\
\hskip3mm $\styc{y}{\varphi\circ\psi}$ \>=\>  $\forall x\;({\bf I}(x,y)\lra\neg\exists z,z'({\bf R}(x,z,z')\wedge\stxb{z}{\varphi}\wedge\stxb{z'}{\psi}))$
\\
$\stxb{x}{\psi\la\varphi}$ \>=\> $\forall z,z'\;(\stxb{z}{\alpha}\;\wedge\;{\bf R}(z',x,z)\;\lra\;\stxb{z'}{\alpha'})$ \\
\hskip3mm $\styc{y}{\psi\la\varphi}$ \>=\> $\forall x\;({\bf I}(x,y)\lra\neg\stxb{x}{\psi\la\varphi})$
\end{tabbing}
\end{table}

\begin{coro}\rm
\label{reduction coro}
For any {\bf SUB}-sentences $\varphi,\psi$, any {\bf sub.ML}$_2$-model (${\mathcal  L}^1$-structure) $\mathfrak{M}=((X,I ,Y,R^{111}),V)$ and corresponding {\bf SUB}-model $\mathfrak{N}$, any $x\in X, y\in Y$,
\begin{enumerate}
\item
$x\forces_\mathfrak{N}\varphi$ iff $x\models_\mathfrak{M}\varphi^\bullet$ iff $\mathfrak{M}\models\stxb{z}{\varphi}[z:=x]$
\item
$\varphi\forces_\mathfrak{N}\psi$  iff $\varphi^\bullet\models_\mathfrak{M}\psi^\bullet$ iff for any $x\in X$, if $\mathfrak{M}\models\stxb{z}{\varphi}[z:=x]$, then $\mathfrak{M}\models\stxb{z}{\psi}[z:=x]$
\item
$\forces_\mathfrak{N}\varphi$ iff $\models_\mathfrak{M}\varphi^\bullet$ iff $\mathfrak{M}\models\forall x\;\stxb{x}{\varphi}$
\item
$\mathfrak{M}\models\stxb{z}{\varphi}[z:=x]$ iff $\mathfrak{M}\models(\forall y\;({\bf I}(z,y)\;\lra\;\exists z'\;({\bf I}(z',y)\;\wedge$\\
\hspace*{7.5cm}$\wedge\;\stxb{z'}{\varphi})[z:=x]$ \\
and similarly for $\styc{v}{\varphi}$.\telos
\end{enumerate}
\end{coro}

Compactness and the L\"{o}wenheim-Skolem properties for any substructural logic {\bf SUB} follow immediately from the reduction presented in this article.
\begin{thm}\rm
\label{results}
The following hold.
\begin{enumerate}
  \item (Compactness) A set of {\bf SUB}-sentences has a model iff every finite subset of it has a model
  \item (Upward L\"{o}wenheim-Skolem) If a set of {\bf SUB}-sentences has a countable model, then it has a model of cardinality $\kappa$, for any $\kappa>\aleph_0$
  \item (Downward L\"{o}wenheim-Skolem) If a set of {\bf SUB}-sentences has an infinite model, then it has a model of cardinality $\aleph_0$
\end{enumerate}
\end{thm}
\begin{proof}
If $\Gamma$ is a set of {\bf SUB}-sentences and $\mathfrak{M}=((X,R,Y,R^{111}),V)$ is a {\bf SUB}-model of $\Gamma$, then clearly every finite subset of $\Gamma$ has a model, to wit $\mathfrak{M}$.
For the converse, if every finite subset $F$ of a set $\Gamma$  of {\bf SUB}-sentences has a model $\mathfrak{M}_F$, then every finite subset of $\Gamma^\bullet$ thereby has a model (the same model $\mathfrak{M}_F$), hence $\Gamma^\bullet$ has a model $\mathfrak{M}=(\mathfrak{F},\imath)$ by compactness of sorted  {\bf FOL} (cf \cite{enderton}) and then $\mathfrak{N}$, on the same frame $\mathfrak{F}$ and with interpretation $V$ defined by $V(p_i)=\lbbox\largediamond\imath(P_i)$, is a model of $\Gamma$.

The L\"{o}wenheim-Skolem properties follow by the same back and forth argument from {\bf SUB}-models to first-order structures, using Corollary \ref{reduction coro} and the L\"{o}wenheim-Skolem theorem for sorted first-order logic.
\end{proof}

\section{Sort Reduction and Decidability}
\label{sort reduction section}
Decidability can be deduced from the finite model property (FMP) of the two-variable fragment of {\bf FOL}, yet to the best of this author's knowledge the FMP property has only been proven for unsorted {\bf FOL}, hence sort-reduction needs to be performed (cf Enderton \cite{enderton}), extending ${\mathcal  L}^1$ with two unary predicates ${\bf T}_1, {\bf T}_2$, redefining first-order structures for ${\mathcal  L}^1$ so that ${\bf T}^I_1=X, {\bf T}^I_2=Y$ and relativising the translation of Table \ref{std-ML2}
{\small
\begin{tabbing}
$\mbox{ST}'_x(P_i)$\hskip0.8cm\==\hskip1mm\= ${\bf T}_1(x)\wedge{\bf P}_i(x)$ \\ $\mbox{ST}'_y(Q_i)$\>=\> ${\bf T}_2(y)\wedge{\bf Q}_i(y)$ \\
$\mbox{ST}'_x(\lbbox \beta)$ \>=\> $\forall y\;({\bf T}_2(y)\;\lra\;({\bf I}(x,y)\lra\mbox{ST}'_y(\beta)))$\\
$\mbox{ST}'_y(\largesquare \alpha)$ \>=\> $\forall x\;({\bf T}_1(x)\;\lra\;({\bf I}(x,y)\lra\mbox{ST}'_x(\alpha)))$\\
$\mbox{ST}'_x(\alpha\odot\alpha')$ \>=\> $\exists z,z'\;({\bf T}_1(z)\wedge{\bf T}_1(z')\wedge{\bf R}(x,z,z')\;\wedge\;\stx{z}{\alpha}\;\wedge\;\stx{z'}{\alpha'})$  \\
$\mbox{ST}'_x(\alpha\rfspoon\alpha')$ \>=\> $\forall z,z'\;({\bf T}_1(z)\wedge{\bf T}_1(z')\lra(\stx{z}{\alpha}\;\wedge\;{\bf R}(z',z,x)\;\lra\;\stx{z'}{\alpha'}))$ \\
$\mbox{ST}'_x(\alpha'\lfspoon\alpha)$ \>=\> $\forall z,z'\;({\bf T}_1(z)\wedge{\bf T}_1(z')\lra(\stx{z}{\alpha}\;\wedge\;{\bf R}(z',x,z)\;\lra\;\stx{z'}{\alpha'}))$
\end{tabbing}
}\noindent
with the obvious clauses for negation and conjunction. For a two-sorted first-order model $\mathfrak{M}$, let $\mathfrak{M}'$ be its corresponding single-sorted model, whose universe is the union of the universes of each sort of $\mathfrak{M}$. The translation by sort reduction is known to be full and faithful (cf Enderton \cite{enderton}).

\begin{thm}[Decidability and Complexity]\rm
A substructural logic {\bf SUB} has the finite model property and its satisfiability problem is decidable in nondeterministic exponential time.
\end{thm}
\begin{proof}
Decidability of the satisfiability problem for {\bf SUB}  follows from the finite model property of the two-variable fragment of first-order logic (assuming here sort-reduction for the standard translation of {\bf SUB}), proven in \cite{FOL-mortimer,FOL-2var,FOL-decision}.

Complexity bounds based on the FMP property are typically exponential. Indeed,  Mortimer \cite{FOL-mortimer} proved the finite model property of the two-variable fragment  with a doubly exponential bound on the model size. The result was improved in \cite{FOL-2var} by Gr\"{a}del, Kolaitis and Vardi to a single-exponential bound. Checking for satisfiability of a sentence in the two-variable fragment of {\bf FOL}, with models of at most exponential size, can be then performed in nondeterministic exponential time, as argued by B\"{o}rger, Gr\"{a}del and Gurevich in \cite{FOL-decision}.
\end{proof}

\section{Conclusions}
\label{conc-section}
This article, a follow-up to \cite{pll7} and based on the representation results of \cite{sdl-exp,discres}, initiated a project of modal translation of non-distributive logics, focusing on a class of familiar non-distributive systems collectively referred to as substructural logics. The modal translation problem for substructural logics has been also addressed by Do\v{s}en \cite{dosen-modal}, who embarked in a project of extending the GMT translation. Do\v{s}en's main result is an embedding of {\bf BCI}, {\bf BCK} and {\bf BCW} with intuitionistic negation into {\bf S4}-type modal extensions of versions of {\bf BCI}, {\bf BCK}, {\bf BCW}, respectively, all equipped with a classical-type negation. The approach, however, does not apply uniformly to substructural logics at large, as pointed out by Do\v{s}en in \cite{dosen-modal}.

In this article, the target modal system is a sorted, residuated modal logic which arises as the logical adjunction of classical propositional logic with the logic of residuated Boolean algebras \cite{residBA}. The result is extendible to the full Lambek-Grishin calculus and to just any non-distributive logic, based on the representation results of \cite{sdl-exp,discr}.

The approach and results in the present article have implications for the correspondence theory of substructural (non-distributive, more generally) logics, which can be addressed by considering the intermediate translation into the companion modal logic. Axioms in the substructural systems correspond to axioms in the modal companion systems, as shown, and frame conditions can be thereby calculated by employing the techniques of automated Shalqvist correspondence for modal logic.


\end{document}